\documentclass[10pt]{article}

\usepackage{amsmath}
\usepackage{graphicx}
\usepackage{float}
\usepackage{enumerate}
\usepackage{amssymb}
\usepackage{amsfonts}
\usepackage{mathtools}
\usepackage{xcolor}
\usepackage{amsmath}
\usepackage{algorithm,algpseudocode}
\usepackage{subcaption}
\usepackage{caption}
\usepackage{verbatim}
\usepackage{bm}
\usepackage{amsthm}
\usepackage[toc,page]{appendix}
\usepackage{hyperref}
\usepackage[export]{adjustbox}
\usepackage{rotating}
\usepackage{multirow}
\usepackage{diagbox}
\usepackage[utf8]{inputenc}
\usepackage[normalem]{ulem}
\usepackage{cleveref}
\usepackage{booktabs,siunitx}

\usepackage{cancel}

\usepackage{pgf,tikz,pgfplots}
\usetikzlibrary{decorations.markings}
\pgfplotsset{compat=1.15}
\usepackage{mathrsfs}
\usetikzlibrary{arrows}

\newcommand{\sign}{\text{sign}}

\DeclareMathOperator{\tr}{tr}
\DeclareMathOperator{\erf}{erf}
\DeclareMathOperator{\Var}{Var}

\newtheorem{theorem}{Theorem}[section]
\newtheorem{corollary}{Corollary}[theorem]
\newtheorem{lemma}[theorem]{Lemma}

\newtheorem*{remark}{Remark}
\theoremstyle{definition}
\newtheorem{definition}{Definition}[section]

\title{Improved variants of the  Hutch++ algorithm for trace estimation\footnote{This work has been supported by the SNSF research project \textit{Fast algorithms from low-rank updates}, grant number: 200020\_178806. Institute of Mathematics, EPF Lausanne, 1015 Lausanne, Switzerland. E-mails: \href{mailto:david.persson@epfl.ch}{david.persson@epfl.ch}, \href{mailto:alice.cortinovis@epfl.ch}{alice.cortinovis@epfl.ch}, \href{mailto:daniel.kressner@epfl.ch}{daniel.kressner@epfl.ch}}}

\author{David Persson\footnotemark[1] \and Alice Cortinovis\footnotemark[1] \and Daniel Kressner\footnotemark[1]}

\begin{document}
\maketitle

\begin{abstract}
This paper is concerned with two improved variants of the Hutch++ algorithm for estimating the trace of a square matrix, implicitly given through matrix-vector products. Hutch++ combines randomized low-rank approximation in a first phase with stochastic trace estimation in a second phase. In turn, Hutch++ only requires $O\left(\varepsilon^{-1}\right)$ matrix-vector products to approximate the trace within a relative error $\varepsilon$ with high probability, provided that the matrix is symmetric positive semidefinite. This compares favorably with the $O\left(\varepsilon^{-2}\right)$ matrix-vector products needed when using stochastic trace estimation alone. In Hutch++, the number of matrix-vector products is fixed a priori and distributed in a prescribed fashion among the two phases. In this work, we derive an adaptive variant of Hutch++, which outputs an estimate of the trace that is within some prescribed error tolerance with a controllable failure probability, while splitting the matrix-vector products in a near-optimal way among the two phases. For the special case of a symmetric positive semi-definite matrix, we present another variant of Hutch++, called Nyström++, which utilizes the so called Nyström approximation and requires only one pass over the matrix, as compared to two passes with Hutch++. We extend the analysis of Hutch++ to Nyström++. Numerical experiments demonstrate the effectiveness of our two new algorithms. 
\end{abstract}

\pagestyle{myheadings}
\thispagestyle{plain}

\section{Introduction}

Computing or estimating the trace of a large symmetric matrix $\bm{A} \in \mathbb{R}^{n \times n}$,
\begin{equation*}
    \tr(\bm{A}): = \sum\limits_{i=1}^n \bm{A}_{ii},
\end{equation*}
is an important problem that arises in a wide variety of applications, such as triangle counting in graphs~\cite{avron}, Frobenius norm estimation~\cite{kressnerbujanovic,gratton}, quantum chromodynamics~\cite{thron1998pade}, computing the Estrada index of a graph~\cite{pena,estrada}, computing the log-determinant \cite{affandi,kressnercortinovis,saibabaipsen,wainwright2006log} and many more. For an excellent overview of the applications to this problem we refer to~\cite{ubaru}.

It can be surprisingly difficult to compute the trace. This difficulty arises if one does not have direct access to the entries of $\bm{A}$, but can only access $\bm{A}$ through matrix-vector products. This appears when, for example, $\bm{A}$ is a function of another matrix $\bm{B}$, such as $\bm{A} = \exp(\bm{B})$, $\bm{A} = \log(\lambda\bm{I} + \bm{B})$, $\bm{A} = \bm{B}^{-1}$ or $\bm{A} = \bm{B}^3$. Computing $\bm{A}$ (or even only its diagonal entries) explicitly in these situations is typically too expensive and may require up to $O(n^3)$ operations. On the other hand, computing (approximate) matrix-vector products $\bm{A}\bm{x}$ is tractable using, for example, Lanczos methods~\cite{higham2008functions,hochbruck}. 

Hutchinson's method~\cite{hutchinson} for trace estimation builds on the following observation: If $\bm{x}$ is a random vector of length $n$ satisfying $\mathbb{E}\bm{x}\bm{x}^T = \bm{I}$ then
\begin{equation*}
    \mathbb{E}\bm{x}^T \bm{A} \bm{x} = \tr(\bm{A}).
\end{equation*}
Therefore, sampling $m$ such quadratic forms and computing the sample mean yields the following unbiased estimator of the trace:
\begin{equation}\label{eq:hutchinson}
    \tr_m(\bm{A}) :=\frac{1}{m}\sum\limits_{i=1}^m \bm{x}_i^T \bm{A} \bm{x}_i = \frac{1}{m} \tr\left(\bm{X}^T \bm{A} \bm{X} \right)\approx \tr(\bm{A}),
\end{equation}
where $\bm{X} = \begin{bmatrix} \bm{x}_1 & \cdots & \bm{x}_m \end{bmatrix}$ contains $m$ independent copies of $\bm{x}$. Common choices for the random vector $\bm{x}$ are standard Gaussians; the entries in $\bm{x}$ are independent identically distributed (i.i.d.) samples from $N(0,1)$, and Rademacher vectors; the entries in $\bm{x}$ are independently chosen to be $-1$ or $+1$ with equal probability. In this work, we choose $\bm{x}$ to be standard Gaussian. In this case, the variance of $\tr_m(\bm{A})$ is given by
\begin{equation}\label{eq:gaussian_variance}
    \Var(\tr_m(\bm{A})) = \frac{2}{m}\|\bm{A}\|_F^2.
\end{equation}

Under the assumption that $\bm{A}$ is symmetric positive semi-definite, one can derive bounds on $m$ that guarantee a small relative error with high probability:
\begin{equation}\label{eq:relerr}
    \mathbb{P}\left(\left|\tr_m(\bm{A}) - \tr(\bm{A})\right| \leq \varepsilon \tr(\bm{A})\right) \geq 1-\delta;
\end{equation}
see, e.g.,~\cite{avron2,gratton, Roosta-K_trace,Roosta-K}. When 
$\bm{A}$ is indefinite, aiming for such a relative bound is unrealistic, as can be easily seen for a non-zero matrix $\bm{A}$ with $\tr(\bm{A}) = 0$. Instead, one aims at deriving bounds on $m$ that guarantee a small \textit{absolute} error: 
\begin{equation}\label{eq:abserr}
    \mathbb{P}\left(|\tr_m(\bm{A}) - \tr(\bm{A})| \leq \varepsilon\right) \geq 1-\delta.
\end{equation}
It is well known that the number of samples needed to attain~\eqref{eq:relerr} or~\eqref{eq:abserr} grows at a rate proportional to $\varepsilon^{-2}$ as $\varepsilon \rightarrow 0$.
To reduce the number of samples (and, in turn, the number of matrix-vector products), different variance reduction techniques were studied~\cite{gambhir,hpp,wu}. These methods aim at finding a decomposition 
\begin{equation}\label{eq:variancereduction}
    \tr(\bm{A}) = \tr(\bm{A}_1) + \tr(\bm{A}_2),
\end{equation}
such that $\tr(\bm{A}_1)$ can be computed explicitly and the stochastic estimator for $\tr(\bm{A}_2)$ has reduced variance, which -- in view of~\eqref{eq:gaussian_variance} -- means that $\bm{A}_2$ has reduced Frobenius norm.
Among these techniques, the Hutch++ algorithm presented in~\cite{hpp} guarantees an $\varepsilon$-relative error, as in \eqref{eq:relerr}, with only $O(\varepsilon^{-1})$ matrix-vector products, provided that $\bm{A}$ is symmetric positive semidefinite. In Hutch++, the matrix $\bm{A}_1$ in~\eqref{eq:variancereduction} is chosen to be a low-rank approximation of $\bm{A}$ obtained with the randomized SVD~\cite{rsvd}, and $\bm{A}_2 = \bm{A}-\bm{A}_1$. The resulting method is presented in Algorithm~\ref{alg:hpp}.
\begin{algorithm}[t]
\caption{Hutch++}
\label{alg:hpp}
\textbf{input:} Symmetric $\bm{A} \in \mathbb{R}^{n \times n}$. Number of matrix-vector products $m \in \mathbb{N}$ (multiple of $3$).\\
\textbf{output:} An approximation to $\tr(\bm{A}): \tr_{m}^{\mathsf{h++}}(\bm{A})$.
\begin{algorithmic}[1]
    \State Sample $\bm{\Omega} \in \mathbb{R}^{n \times \frac{m}{3}}$ with i.i.d. $N(0,1)$ or Rademacher entries.
    \State Compute $\bm{Y} = \bm{A}\bm{\Omega}$ \label{line:Y}.
    \State Get an orthonormal basis $\bm{Q} \in \mathbb{R}^{n \times \frac{m}{3}}$ for $\text{range}(\bm{Y})$.
    \State Sample $\bm{\Psi} \in \mathbb{R}^{n \times \frac{m}{3}}$ with i.i.d. $N(0,1)$ or Rademacher entries.
    \State \textbf{return} $\tr_m^{\mathsf{h++}}(\bm{A}) = \text{tr}(\bm{Q}^T \bm{A} \bm{Q}) + \frac{3}{m} \tr(\bm{\Psi}^T(\bm{I} - \bm{Q}\bm{Q}^T) \bm{A} (\bm{I} - \bm{Q}\bm{Q}^T) \bm{\Psi})$ \label{line:return}
\end{algorithmic}
\end{algorithm}
Hutch++ consists of two phases. The first phase is concerned with obtaining a low-rank approximation $\bm{A} \approx \bm{Q}\bm{Q}^T \bm{A}$ and exploits the cyclic property of the trace: $\tr(\bm{Q}\bm{Q}^T \bm{A}) = \tr(\bm{Q}^T \bm{A}\bm{Q})$. It uses $\frac{2m}{3}$ matrix-vector products with $\bm{A}$: $\bm{A}\bm{\Omega}$ in line~\ref{line:Y} of Algorithm~\ref{alg:hpp} and $\bm{AQ}$ to compute $\tr(\bm{Q}^T \bm{A}\bm{Q})$ in line~\ref{line:return}. The second phase is concerned with estimating $\tr(\bm{A}-\bm{Q}\bm{Q}^T \bm{A}) = \tr((\bm{I} - \bm{Q}\bm{Q}^T) \bm{A} (\bm{I} - \bm{Q}\bm{Q}^T))$ via the stochastic trace estimator~\eqref{eq:hutchinson}. It uses the remaining $\frac{m}{3}$ matrix-vector products with $\bm{A}$ to compute $\bm{A} ( (\bm{I} - \bm{Q}\bm{Q}^T)\bm{\Psi})$ in line~\ref{line:return} of Algorithm~\ref{alg:hpp}.

\subsection{Contributions}


The effectiveness of the two phases of Hutch++ depends on the singular values of $\bm{A}$. When $\bm{A}$ admits an accurate low-rank approximation (e.g., when its singular values decay quickly), it would be sufficient to perform the approximation $\tr(\bm{A}) \approx \tr(\bm{A}_1)$, as suggested by~\cite{saibabaipsen} and skip the second phase of Hutch++. On the other hand, when all singular values of $\bm{A}$ are nearly equal, the variance reduction achieved during the first phase of Hutch++ is insignificant and all effort should be spent on the second phase, the stochastic trace estimator~\eqref{eq:hutchinson}. One can easily perceive a situation where it is preferable to spend maybe not all but most of the matrix-vector products on the stochastic trace estimator.
Algorithm~\ref{alg:hpp} does not recognize such situations; the number of matrix-vector products is fixed a priori and distributed in a prescribed fashion among the two phases.

Furthermore, the results in \cite{hpp} are of significant theoretical importance, but since the $O(\varepsilon^{-1})$ bound comes without explicit constants it gives practitioners little indication of how many matrix-vector products to use when estimating the trace of a given matrix $\bm{A}$. One can work out the constants, for example by using results in \cite{rsvd} if Gaussian random vectors are used, and conclude that, for fixed failure probability $\delta$, $m = C/\varepsilon$ matrix-vector products are sufficient to get an estimate of the trace with a relative error at most $\varepsilon$ with high probability, where $C$ is a constant depending only on $\delta$. However, this bound is in some cases a significant overestimation of the number of required matrix-vector products. To see this, consider the case when $\bm{A}$ has rapidly decaying singular values. In this case it would be sufficient to perform the approximation $\tr(\bm{A}) \approx \tr(\bm{A}_1)$, with potentially much fewer matrix-vector products than suggested by the $C/\varepsilon$ bound. On the other hand, when all singular values of $\bm{A}$ are nearly equal, the standard deviation of the stochastic trace estimator, which is proportional to $\|\bm{A}\|_F$, is much smaller than $\tr(\bm{A})$. Therefore, the relative error of the estimate produced by the stochastic trace estimator with only a few matrix-vector products, potentially much fewer than suggested by the $C/\varepsilon$ bound, will give a sufficiently accurate estimate of the trace with high probability.\footnote{To see this, recall that the standard deviation of the stochastic trace estimator with $m$ samples equals $\sqrt{2/m}\|\bm{A}\|_F$. This can be much smaller than $\varepsilon \tr(\bm{A})$ with $m$ potentially much smaller than $C/\varepsilon$, provided $\varepsilon$ is not too small.}

In this work, we develop an adaptive version of Hutch++ to address the above mentioned issues. We start with developing a prototype algorithm which given a prescribed tolerance $\varepsilon$ and failure probability $\delta$ outputs an estimate of the trace of $\bm{A}$, denoted $\tr_{\mathsf{adap}}(\bm{A})$, such that 
\begin{equation}\label{eq:adaperr}
    |\tr_{\mathsf{adap}}(\bm{A}) - \tr(\bm{A})| \leq \varepsilon
\end{equation}
holds, provably, with probability at least $1-\delta$. At the same time, our algorithm attempts to minimize the overall number of matrix-vector products by distributing them between the two phases in a near-optimal fashion. Then we modify the prototype algorithm to develop a more efficient adaptive trace estimation algorithm, which will be A-Hutch++. Note, however, that the potential for improving Hutch++ is limited, in~\cite{hpp} the $O(\varepsilon^{-1})$ bound mentioned above is proven to be optimal up to a $\log(\varepsilon^{-1})$ factor. In practice, we observe that our adaptive version of Hutch++ is never worse than the original Hutch++ and often outperforms it. Possibly more importantly, the output of our prototype algorithm comes with a probabilistic guarantee on the error of the estimate of $\tr(\bm{A})$ without requiring the user to know a priori how many matrix-vector products are needed. Our algorithm does not assume that $\bm{A}$ is positive definite, which is why we focus on estimating $\tr(\bm{A})$ up to a given absolute error.

Another aspect we address in this work is that  the Hutch++ algorithm requires several passes over the matrix $\bm{A}$; in Algorithm~\ref{alg:hpp} the matrix-vector products carried out in line 5 depend on earlier ones. In the streaming model it is desirable to design an algorithm that requires only one pass over $\bm{A}$ and if the matrix of interest is modified by a linear update $\bm{A} + \bm{E}$ one does not have to revisit $\bm{A}$ to update the output of the algorithm. Such a single pass property also increases parallelism. A single pass trace estimation algorithm was presented in~\cite{hpp} and we will call it \textit{Single Pass Hutch++} in this work. For a symmetric positive semidefinite matrix this algorithm comes with nearly the same theoretical guarantees as Hutch++, but performs worse in practice. In the case of a symmetric positive semidefinite matrix we develop a variation of Hutch++, Nyström++, utilizing the Nyström approximation~\cite{gittensmahoney}. Nyström++ requires only one pass over $\bm{A}$ and satisfies, up to constants, the theoretical guarantees of Hutch++. This new variation of Hutch++ significantly outperforms Single Pass Hutch++ and often outperforms Hutch++. 
\begin{remark}
Note that the word adaptive is used differently in~\cite{hpp}, where Hutch++ itself is already called adaptive because the matrix-vector products $\bm{AQ}$ depend on (and thus adapt to) the previously computed $\bm{A\Omega}$. In this work, we follow the convention where the term adaptive refers to an algorithm that adapts to a desired error bound. The Single Pass Hutch++ mentioned above is called
NA-Hutch++ (non-adaptive variant of Hutch++) in~\cite{hpp}.
\end{remark}

\subsection{Notation} \label{sec:notation}
For a vector $\bm{x} \in \mathbb{R}^n$ we let $\|\bm{x}\|_2 = \big(\sum\limits_{i=1}^n x_i^2\big)^{1/2}$ denote the Euclidean norm of $\bm{x}$. We let $\sigma_1 \geq \sigma_2 \geq \cdots \geq \sigma_n \geq 0$ denote the singular values of $\bm{A}$. Thus, we have $\|\bm{A}\|_2 = \sigma_1$ and $\|\bm{A}\|_F^2 = \sigma_1^2 + \cdots + \sigma_n^2$. The nuclear norm of $\bm{A}$ is defined as $\|\bm{A}\|_* = \sigma_1 + \cdots + \sigma_n$. We let $\rho(\bm{A}) = \frac{\|\bm{A}\|_F^2}{\|\bm{A}\|_2^2}$ denote the stable rank of $\bm{A}$. Furthermore, for a matrix $\bm{B} \in \mathbb{R}^{m \times p}$, $p \geq m$, with linearly independent rows we let $\bm{B}^{\dagger} := \bm{B}^T(\bm{B}\bm{B}^T)^{-1}$ denote the Moore-Penrose pseudoinverse. We say that a random $n \times k$ matrix $\bm{\Omega}$ with i.i.d. $N(0,1)$ entries is a \textit{standard Gaussian matrix}. In the case of $k = 1$ we say that it is a \textit{standard Gaussian vector}. 


\section{Adaptive variants of Hutch++}\label{section:adap_hpp}

The aim of this section is to develop adaptive variants of Hutch++  (Algorithm~\ref{alg:hpp}). In a first step, we derive a prototype algorithm that aims at minimizing the number of matrix-vector products and comes with a guaranteed bound on the failure probability. The latter requires to estimate the variance or, equivalently (see~\eqref{eq:gaussian_variance}), the Frobenius norm, and this estimate needs additional matrix-vector products. Our final algorithm A-Hutch++ reuses these 
matrix-vector products for trace estimation and chooses the number of them in an adaptive fashion. In turn, this creates dependencies that complicate the analysis but do not lead to observed failure probabilities that are above the prescribed failure probability.

\subsection{Derivation of adaptive Hutch++}

The first phase of Algorithm~\ref{alg:hpp} requires $2r$ matrix-vector products with $\bm{A}$ to obtain a rank-$r$ approximation $\bm{Q}^{(r)}\bm{Q}^{(r)T}\bm{A}$, where we have added a superscript to emphasize the dependence on $r$.
Let $M(r)$ be the number of matrix-vector products with $\bm{A}$ in the second phase such that the stochastic trace estimator of 
\begin{equation} \label{eq:defAr}
     \bm{A}^{(r)}_{\text{rest}} := (\bm{I}-\bm{Q}^{(r)}\bm{Q}^{(r)T})\bm{A}(\bm{I}-\bm{Q}^{(r)}\bm{Q}^{(r)T})
\end{equation}
attains a prescribed accuracy and success probability. Then the total number of matrix-vector products with $\bm{A}$ is \begin{equation}\label{eq:matvecs}
    m(r) = 2r + M(r).
\end{equation}
We aim at minimizing $m(r)$ in order to obtain a near-optimal distribution of matrix-vector products between the two phases.\footnote{In practice we perform randomized low-rank approximations. Consequently, $\bm{A}^{(r)}_{\text{rest}}$ is random and therefore the function $m$ is a random variable. Hence, it can be ambiguous what it means to minimize $m$. To clarify this, first note that we always assume $r\leq n$, where $\bm{A}$ is $n \times n$, since when $r = n$ we are able to exactly compute $\tr(\bm{A})$. Therefore, we will never sample more than $n$ random vectors to obtain a low-rank approximation. Thus, let $\bm{\Omega} \in \mathbb{R}^{n \times n}$ be the random matrix from which we can construct $\bm{Q}^{(1)},\bm{Q}^{(2)},\ldots,\bm{Q}^{(n)}$. Conditioned on $\bm{\Omega}$ the function $m$ becomes deterministic and has a minimum, which is what we aim to find. We will describe a heuristic strategy to find the minimum in Section~\ref{sec:minimum}.} For this purpose, we first derive a suitable expression for $M(r)$.

\subsubsection{Analysis of trace estimation}\label{sec:analysis_ahpp}

The tightest tail bound available in the literature for the stochastic trace estimator $\tr_m(\bm{B})$ for a symmetric matrix $\bm{B}$ is~\cite[Theorem 1]{kressnercortinovis}, which states that
\begin{equation}\label{eq:tailbound_kressnercortinovis}
    \mathbb{P}\left(|\tr_m(\bm{B})-\tr(\bm{B})| \geq \varepsilon \right) \leq 2\exp\left(-m \frac{\varepsilon^2}{4\|\bm{B}\|_F^2 + 4\varepsilon\|\bm{B}\|_2}\right).
\end{equation}
In most situations of interest, the term involving $\|\bm{B}\|_2$ will be insignificant. The following lemma is a variation of~\eqref{eq:tailbound_kressnercortinovis} that suppresses this term for sufficiently large $m$, similar to~\cite[Lemma 2.1]{hpp}. We note in passing that \eqref{eq:tailbound_kressnercortinovis} as well as the following lemma can be improved; see Appendix \ref{appendix:tailbound}. 
\begin{lemma}\label{lemma:combinedbound_gaussian}
Given $\ell > 0$ assume that $m \geq \frac{4(1+\ell)\log\left(2/\delta\right)}{\ell^2 \rho(\bm{B})}$. Then the inequality
\begin{equation} \label{eq:combinedbound_gaussian}
    |\tr_m(\bm{B})-\tr(\bm{B})| \leq 2 \sqrt{1+\ell} \sqrt{\frac{\log\left(2/\delta\right)}{m}}\|\bm{B}\|_F
\end{equation}
holds with probability at least $1-\delta$.
\end{lemma}
\begin{proof}
Inserting the right-hand side of~\eqref{eq:combinedbound_gaussian}, $
    \varepsilon := 2\sqrt{1+\ell} \sqrt{\frac{\log(2/\delta)}{m}}\|\bm{B}\|_F
$ , into~\eqref{eq:tailbound_kressnercortinovis} one obtains the desired result:
\begin{eqnarray*}
 \mathbb{P}\left(|\tr_m(\bm{B})-\tr(\bm{B})| \geq \varepsilon \right) & 
\le    & 2\exp\left(-\frac{(1+\ell) \log(2/\delta) \|\bm{B}\|_F}{\|\bm{B}\|_F + 2\sqrt{1+\ell} \sqrt{\frac{\log(2/\delta)}{m}} \|\bm{B}\|_2}\right) \\
&\le & 
2\exp\left(-\frac{(1+\ell) \log(2/\delta) \|\bm{B}\|_F}{(1+\ell)\|\bm{B}\|_F }  \right) = \delta,
\end{eqnarray*}
where the second inequality utilizes 
\begin{equation*}
    \ell \|\bm{B}\|_F \geq 2\sqrt{1+\ell} \sqrt{\frac{\log(2/\delta)}{m}}\|\bm{B}\|_2,
\end{equation*}
a consequence of the assumption on $m$. 
\end{proof}
Let
\begin{equation}\label{eq:definition_C}
 C(\varepsilon,\delta):= 4(1+\ell)\varepsilon^{-2}\log(2/\delta).
\end{equation}
By Lemma~\ref{lemma:combinedbound_gaussian}, for sufficiently small $\varepsilon$, $C(\varepsilon,\delta)\|\bm{B}\|_F^2$ samples
are sufficient to achieve
$
    |\tr_m(\bm{B})-\tr(\bm{B})| \leq \varepsilon
$
with probability at least $1-\delta$. In practice one cannot assume to know, or be able to compute, the stable rank appearing in the condition $m \geq \frac{4(1+\ell) \log(2/\delta)}{\ell^2\rho(\bm{B})}$. Since the stable rank is always larger than 1, requiring $m \geq \frac{4(1+\ell) \log(2/\delta)}{\ell^2}$ would be sufficient to ensure that $m \geq \frac{4(1+\ell) \log(2/\delta)}{\ell^2\rho(\bm{B})}$. However, in practice we set $\ell = 0$ and completely omit the side condition $m \geq \frac{4(1+\ell) \log(2/\delta)}{\ell^2\rho(\bm{B})}$. While not justified by Lemma~\ref{lemma:combinedbound_gaussian}, we observe no significant loss in the success probabilities of our algorithm, see Section \ref{section:synthetic_matrices}.

\subsubsection{Finding the minimum of $m(r)$}\label{sec:minimum}

Applying the results above to $\bm{B} = \bm{A}^{(r)}_{\text{rest}}$ implies that a suitable choice for the function $m(r)$ in~\eqref{eq:matvecs} is given by
\begin{equation}\label{eq:N}
    m(r) = 2r + C(\varepsilon,\delta) \|\bm{A}^{(r)}_{\text{rest}}\|_F^2.
\end{equation}

In the idealistic scenario that $\bm{Q}^{(r)}$ contains the dominant $r$ singular vectors, we have $\|\bm{A}^{(r)}_{\text{rest}}\|_F^2 = \sigma_{r+1}^2 + \cdots + \sigma_n^2$. This implies that the differences $m(r)-m(r-1) = 2 - C(\varepsilon,\delta) \sigma_{r+1}^2$ are monotonically increasing and switch sign at most once. In turn, $r^*$ is a global minimum whenever it is a local minimum, that is, $m(r^*\pm1) \ge m(r^*)$.
Since $\bm{Q}^{(r)}$ only approximates the space spanned by the dominant $r$ singular vectors of $\bm{A}$, these relations are not guaranteed to hold. In practice, we have observed $m(r^*\pm1) \ge m(r^*)$ to remain a reliable criterion; see Figure~\ref{fig:minimum} for an example.
\begin{figure}
    \centering
    \includegraphics[max size = {12cm}{5cm}]{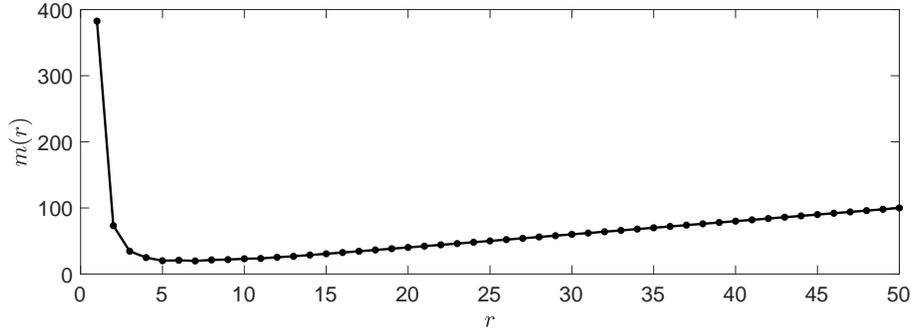}
    \caption{In this example we let $\bm{A} = \bm{U} \bm{\Lambda} \bm{U}^T \in \mathbb{R}^{1000 \times 1000}$ where $\bm{U}$ is a random orthogonal matrix and $\bm{\Lambda}$ is a diagonal matrix with $\bm{\Lambda}_{ii} = 1/i^2$. The x-axis shows the rank $r$, and the y-axis is the function $m(r)$ defined in \eqref{eq:N} with $\delta = 0.01$, $\varepsilon = 0.05\tr(\bm{A})$ and $\ell = 0$. The function has its minimum at $r^* = 7$.}
    \label{fig:minimum}
\end{figure}

Evaluating $m(r)$ involves the quantity $\|\bm{A}_{\text{rest}}^{(r)}\|_F^2$, which is too expensive to evaluate. Using the symmetry of $\bm{A}$ and the unitary invariance of the Frobenius norm we get
\begin{align}
   &\|\bm{A}^{(r)}_{\text{rest}}\|_F^2 = \|\bm{A}\|_F^2 + \|\bm{Q}^{(r)T}\bm{A}\bm{Q}^{(r)}\|_F^2 - 2 \|\bm{A}\bm{Q}^{(r)}\|_F^2\label{eq:inprob_expansion}.
\end{align}
In turn, $m(r)$ and the function
\begin{equation}\label{eq:tildeN}
    \tilde{m}(r) := 2r + C(\varepsilon,\delta)\big(\|\bm{Q}^{(r)T}\bm{A}\bm{Q}^{(r)}\|_F^2 - 2 \|\bm{A}\bm{Q}^{(r)}\|_F^2\big)
\end{equation}
have the same minimum. The latter can be cheaply computed by recursive updating, without any additional matrix-vector products with $\bm{A}$.

To summarize, we adapt the randomized SVD to build $\bm{Q}^{(r)}$ column-by-column, similar to as described in \cite[Section 4.4]{rsvd}, and stop the loop whenever a minimum of $\tilde{m}(r)$ is detected. By the heuristics discussed above, it is safe to stop at $r = r^*$ when $\tilde{m}(r^*) > \tilde{m}(r^*-1) > \tilde{m}(r^*-2)$.

\subsubsection{Estimating the Frobenius norm of the remainder}

Having found an approximate minimum $r^*$ of $\tilde{m}(r)$ and computed $\bm{Q} \equiv \bm{Q}^{(r^*)}$, it remains to apply stochastic trace estimation to $\bm{A}_{\text{rest}}\equiv \bm{A}^{(r^*)}_{\text{rest}}$.
By Lemma~\ref{lemma:combinedbound_gaussian} it suffices to use $M \ge C(\varepsilon,\delta)\|\bm{A}_{\text{rest}}\|_F^2$ samples. Because computing $\|\bm{A}_{\text{rest}}\|_F$ is too expensive, we need to resort (once more) to a stochastic estimator utilizing only matrix-vector products. The following result is essential for that purpose.
\begin{lemma}\label{lemma:frobenius_estimation} Let $\bm{\Omega} \in \mathbb{R}^{n \times k}$ be a standard Gaussian matrix and let $\bm{B} \in \mathbb{R}^{n \times n}$. For any $\alpha \in (0,1)$ it holds that
\begin{align*}
    \mathbb{P}\left(\frac{1}{k}\|\bm{B}\bm{\Omega}\|_F^2 < \alpha \|\bm{B}\|_F^2 \right) \leq \mathbb{P}(X < \alpha)=\frac{\gamma(k/2,\alpha k/2)}{\Gamma(k/2)},
\end{align*}
where $X \sim \Gamma(k/2,k/2)$ (gamma distribution with shape and rate parameter $k/2$),
$\gamma(s,x) := \int_0^x t^{s-1}e^{-t} dt$ is the lower incomplete gamma function and $\Gamma(s)$ is the standard gamma function.
\end{lemma}
\begin{proof} 
It is well known that 
\begin{equation}\label{eq:chi2_sum}
    \frac{1}{k} \|\bm{B\Omega}\|_F^2 = \frac{1}{k}\sum\limits_{j=1}^n \sigma_j^2 Z_j,
\end{equation}
where $Z_j$, $j = 1,\ldots,n$, denote i.i.d. $\chi^2_k$ random variables; see, e.g., \cite[Section 2]{gratton}.
Setting $X_j := \frac{1}{k}Z_j \sim \Gamma(k/2,k/2)$ and $\lambda_j = \frac{\sigma_j^2}{\|\bm{B}\|_F^2}$ for $j = 1,\ldots,n$ we rewrite
\begin{equation}\label{eq:gamma_sum}
    \mathbb{P}\Big(\frac{1}{k} \|\bm{B}\bm{\Omega}\|_F^2 < \alpha \|\bm{B}\|_F^2 \Big) = \mathbb{P}\bigg(\sum\limits_{j=1}^n \lambda_j X_j < \alpha \bigg).
\end{equation}
By~\cite[Theorem 2.2]{Roosta-K} the right-hand side is bounded for every $\alpha \in (0,1)$ by $\mathbb{P}\left(X_1 < \alpha\right)$, which completes the proof. 
\end{proof}

Lemma~\ref{lemma:frobenius_estimation} states that if $\frac{\gamma(k/2,\alpha k/2)}{\Gamma(k/2)} \leq \delta$ then $\frac{1}{k\alpha}\|\bm{B}\bm{\Omega}\|_F^2 >  \|\bm{B}\|_F^2$ with probability at least $1-\delta$. Hence, using $M := \lceil C(\varepsilon,\delta) \cdot \frac{1}{k\alpha} \|\bm{A}_{\text{rest}}\bm{\Omega}\|_F^2 \rceil$ samples ensures an error of at most $\varepsilon$ with low failure probability. See Figure~\ref{fig:alpha} for the relationship between $k,\alpha$ and $\delta$.

\begin{figure}
    \centering
    \includegraphics[max size = {12cm}{5cm}]{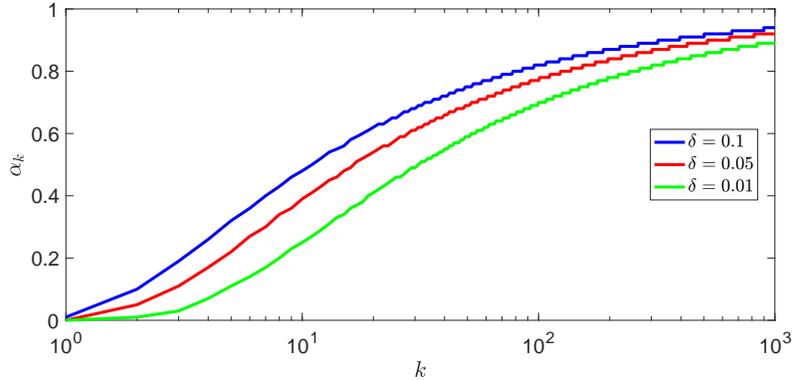}
    \caption{For different choices of $\delta$, this plot demonstrates the relationship between $k$ and the largest choice of $\alpha$ such that $\frac{\gamma(k/2,\alpha k/2)}{\Gamma(k/2)} \leq \delta$.}
    \label{fig:alpha}
\end{figure}

\subsubsection{A prototype algorithm}

Combining the results presented above we obtain the prototype algorithm presented in Algorithm~\ref{alg:adap_hpp}.  To reduce the number of passes over the matrix $\bm{A}$ the algorithm can be implemented in a block-wise fashion, which can in turn lead to a reduction  of wall-clock time. For block-size $b=1$ we use the heuristic stopping criteria for the low-rank approximation described above. For larger block-sizes it is sufficient to use $m(r^*-b) < m(r^*)$ as a stopping criteria.
\begin{algorithm}[H]
\caption{Prototype algorithm}
\label{alg:adap_hpp}
\textbf{input:} Symmetric $\bm{A} \in \mathbb{R}^{n \times n}$. Error tolerance $\varepsilon > 0$. Failure probability $\delta \in (0,1)$. Parameter $\ell > 0$. Block-size $b$.\\
\textbf{output:} An approximation to $\tr(\bm{A}):\tr_{\mathsf{adap}}(\bm{A})$.
\begin{algorithmic}[1]
\\ $\bm{Y}^{(b)} = \bm{A}\bm{\Omega}^{(b)}$ where $\bm{\Omega}^{(b)} \in \mathbb{R}^{n \times b}$ has i.i.d. $N(0,1)$ entries. \label{line:rsvdbegin}
\\ Obtain orthonormal basis $\widehat{\bm{Q}}^{(b)}$ for $\text{range}\left(\bm{Y}^{(b)}\right)$.
\\ $\bm{Q}^{(1)} = \widehat{\bm{Q}}^{(1)}$
\\ $\text{trest}_1 = \tr\left( \widehat{\bm{Q}}^{(1)T}\left(\bm{A} \widehat{\bm{Q}}^{(1)}\right)\right)$
\\ Compute $\tilde{m}(b)$.
\\ $r = b$
\While{A minimum of $\tilde{m}(r)$ not detected}
\\ \quad $\bm{Y}^{(r + b)} = \bm{A}\bm{\Omega}^{(r+b)}$ where $\bm{\Omega}^{(r+b)} \in \mathbb{R}^{n \times b}$ has i.i.d. $N(0,1)$ entries.
\\ \quad $\widetilde{\bm{Q}}^{(r+b)} = (\bm{I}-\bm{Q}^{(r)}\bm{Q}^{(r)T})\bm{Y}^{(r+b)}$
\\ \quad Obtain orthonormal basis $\widehat{\bm{Q}}^{(r+b)}$ for $\text{range}\left(\widetilde{\bm{Q}}^{(r+b)}\right)$.
\\ \quad $\bm{Q}^{(r+b)} = \begin{bmatrix} \bm{Q}^{(r)} & \widehat{\bm{Q}}^{(r+b)} \end{bmatrix}$
\\ \quad $\text{trest}_1 = \text{trest}_1 + \tr\left(\widehat{\bm{Q}}^{(r+b)T} \left(\bm{A} \widehat{\bm{Q}}^{(r+b)}\right) \right)$
\\ \quad Update $\tilde{m}(r+b)$ recursively.
\\ \quad $r = r+b$
\EndWhile
\\ Let $\bm{Q} = \bm{Q}^{(r)}$ and $\bm{A}_{\text{rest}} = (\bm{I}- \bm{Q}\bm{Q}^T)\bm{A}(\bm{I}-\bm{Q}\bm{Q}^T)$. \label{line:rsvdend} \Comment{$\bm{A}_{\text{rest}}$ is never formed explicitly.}
\\ Choose $(k,\alpha) \in \mathbb{N} \times (0,1)$ such that $\frac{\gamma\left(k/2,\alpha k/2\right)}{\Gamma\left(k/2\right)} \leq \delta$. \label{line:tuple}
\\ $M = \max\left\{\frac{4(1+\ell)\log(2/\delta)}{\ell^2},\lceil C(\varepsilon,\delta) \cdot \frac{1}{k\alpha} \|\bm{A}_{\text{rest}}\bm{\Psi} \|_F^2 \rceil\right\}$ where $\bm{\Psi} \in \mathbb{R}^{n \times k}$ is a standard Gaussian matrix \label{line:frobenius_estimation}.
\\ $\text{trest}_2 = \tr_M(\bm{A}_{\text{rest}})$ \label{line:trace_estimation}
\\ \textbf{return} $\tr_{\mathsf{adap}}(\bm{A}) = \text{trest}_1 + \text{trest}_2$
\end{algorithmic}
\end{algorithm} 
A simple probabilistic analysis yields the following result on the success probability of Algorithm~\ref{alg:adap_hpp}:
\begin{lemma}\label{lemma:success1}
The output of Algorithm~\ref{alg:adap_hpp} satisfies 
$|\tr_{\mathsf{adap}}(\bm{A}) - \tr(\bm{A})| \leq \varepsilon$ with probability at least $1-2\delta$. 
\end{lemma}
\begin{proof}
For the moment, let us consider $\bm{Q}$ fixed and, hence, $\bm{A}_{\text{rest}}$ deterministic.
For a fixed arbitrary integer $N$ let us consider the event
\begin{equation*}
    S_N := \left\{|\tr_N(\bm{A}_{\text{rest}}) - \tr(\bm{A}_{\text{rest}})| \leq \varepsilon \right\}.
\end{equation*}
Let $M$ be the random variable defined in line~\ref{line:frobenius_estimation} of Algorithm~\ref{alg:adap_hpp}. Therefore, $S_M$ is the event that the estimate of $\tr(\bm{A}_{\text{rest}})$ from Algorithm~\ref{alg:adap_hpp} has an error at most $\varepsilon$. That is,
\begin{equation*}
    S_M = \left\{|\tr_M(\bm{A}_{\text{rest}}) - \tr(\bm{A}_{\text{rest}})| \leq \varepsilon \right\} = \bigcup\limits_{N \geq 1} \left[S_N \cap \{M = N\}\right]
\end{equation*}
The analysis of $\mathbb{P}(S_M)$ is complicated by the fact that the integer $M$ defined in line~\ref{line:frobenius_estimation} of Algorithm~\ref{alg:adap_hpp} is also random.
Letting \[
M_1 := \max\left\{\frac{4(1+\ell)\log(2/\delta)}{\ell^2\rho(\bm{A}_{\text{rest}})},C(\varepsilon,\delta)\|\bm{A}_{\text{rest}}\|_F^2\right\},
    \]
we know from Lemma~\ref{lemma:frobenius_estimation} that $\mathbb{P}(M \geq M_1) \geq 1-\delta$ and from~\eqref{eq:definition_C} that $\mathbb{P}(S_N) \geq 1-\delta$ for $N \geq M_1$. Moreover, it is important to remark that the events $S_N$ and $M=N$ are independent. In particular, this implies $\mathbb{P}(S_M|M = N) = \mathbb{P}(S_N)$. Combining these observations yields

\begin{eqnarray*}
    \mathbb{P}(S_M) & \geq & \mathbb{P}( S_M  \cap \{M \ge M_1 \} ) \\
   & =&\sum\limits_{N \geq M_1} \mathbb{P}(S_{M} \cap \{M = N\}) 
     = \sum\limits_{N \geq M_1} \mathbb{P}(S_{M} | M = N)\mathbb{P}(M=N)\\
    & = & \sum\limits_{N \geq M_1} \mathbb{P}(S_N)\mathbb{P}(M = N) 
    \geq (1-\delta)\sum\limits_{N \geq M_1} \mathbb{P}(M = N) \\
    &= &(1-\delta)\mathbb{P}(M \geq M_1) \geq (1-\delta)^2 \geq 1-2\delta,
\end{eqnarray*}
which holds independently of $\bm{Q}$ and thus completes the proof. 
\end{proof}

\subsection{A-Hutch++}\label{section:heuristic}

To turn Algorithm~\ref{alg:adap_hpp} into a practical method, we need to address the choice of the pair $(k,\alpha)$ in line~\ref{line:tuple} and apply further modification to increase its efficiency by reusing the matrix vector products in the Frobenius norm estimation in line \ref{line:frobenius_estimation} in the trace estimation in line \ref{line:trace_estimation} of Algorithm \ref{alg:adap_hpp}. 

For fixed $k$, it makes sense to choose $\alpha$ as large as possible because $M$ decreases with increasing $\alpha$; see line~\ref{line:frobenius_estimation}. Thus, we set
\begin{equation}\label{eq:alpha_k}
    \alpha_k := \sup\left\{\alpha \in (0,1): \frac{\gamma\left(k/2,\alpha k/2\right)}{\Gamma\left(k/2\right)} \leq \delta\right\}.
\end{equation}
\begin{lemma}\label{lemma:monotonicity}
The sequence $\{\alpha_k\}_{k \in \mathbb{N}}$ defined by~\eqref{eq:alpha_k} increases monotonically and converges to $1$.
\end{lemma}
\begin{proof}
Letting $X  := \frac{1}{k} \sum\limits_{i=1}^k X_i \sim \Gamma(k/2,k/2)$ for i.i.d. $\chi_1^2$ random variables $X_i$, we set
\begin{equation*}
    p_k(\alpha) := \mathbb{P}\left( X \leq \alpha \right) = \frac{\gamma\left(k/2,\alpha k/2\right)}{\Gamma\left(k/2\right)}.
\end{equation*}
By~\cite[Theorem 2.1]{Roosta-K} $p_{k+1}(\alpha) \leq p_k(\alpha)$ for every $\alpha \in (0,1]$. Furthermore, by continuity of $p_k$ in $\alpha$ and monotonicity of $p_k(\alpha)$ in $k$ we have
\begin{equation*}
    \delta = p_k(\alpha_k) = p_{k+1}(\alpha_{k+1}) \leq p_{k}(\alpha_{k+1}).
\end{equation*}
Thus, by monotonicity of $p_k$ in $\alpha$ we have $\alpha_k \leq \alpha_{k+1}$, which proves the monotonicity of the sequence $\{\alpha_k\}_{k \in \mathbb{N}}$.

To show $\alpha_k \rightarrow 1$ as $k \rightarrow +\infty$, let $\alpha_\varepsilon := 1-\varepsilon > 0$ for fixed arbitrary $0 < \varepsilon <1$.
By the law of large numbers, $p_k(\alpha_{\varepsilon}) \rightarrow 0$ and by the argument above this convergence is monotonic. Let $k_{\varepsilon,\delta} = \min\{k \in \mathbb{N} : p_k(\alpha_{\varepsilon}) \leq \delta \}$. Let $k \geq k_{\varepsilon,\delta}$. Then, $\delta \geq p_{k_{\varepsilon,\delta}}(\alpha_{\varepsilon}) \geq p_k(\alpha_{\varepsilon})$. Thus, for all $k \geq k_{\varepsilon,\delta}$ we have $1 \geq \alpha_k \geq \alpha_{\varepsilon} \geq 1-\varepsilon$, as required. 
\end{proof}

Furthermore, define the following random sequence $M_k$:
\begin{equation}\label{eq:M_k}
    M_k := C(\varepsilon,\delta) \cdot \frac{1}{k\alpha_k} \|\bm{A}_{\text{rest}}\bm{\Psi}^{(k)}\|_F^2, \quad \bm{\Psi}^{(k)} = \begin{bmatrix} \bm{\Psi}^{(k-1)} & \bm{\psi}^{(k)} \end{bmatrix}, \quad \bm{\psi}^{(k)} \sim N(\bm{0},\bm{I}).
\end{equation}
By the law of large numbers we have $M_k \rightarrow C(\varepsilon,\delta) \|\bm{A}_{\text{rest}}\|_F^2$ almost surely as $k \rightarrow +\infty$. If we reuse the matrix vector products from line \ref{line:frobenius_estimation} in line \ref{line:trace_estimation} the total number of performed matrix vector products in the second phase of Algorithm \ref{alg:adap_hpp} is
\begin{equation}\label{eq:reuse_matvecs}
    \max\left\{k, \lceil M_k \rceil \right\}.
\end{equation}
Because of the monotonicity of $\alpha_k$, and as seen in Figure~\ref{fig:adaptive_M}, $M_k$ is expected to decrease in $k$. Hence, in order to minimize \eqref{eq:reuse_matvecs} we choose $k$ such that $k = \lceil M_k \rceil$. Thus, we evaluate $M_k$ inside a while loop and stop the while loop once we detect $k > M_k$ for the first time. At this point we reuse the computation $\bm{A}_{\text{rest}}\bm{\Psi}^{(k)}$ to estimate $\tr(\bm{A}_{\text{rest}})$. The resulting algorithm is presented in Algorithm~\ref{alg:adap_hpp_heuristic}. As with the prototype algorithm, Algorithm~\ref{alg:adap_hpp_heuristic} can also be implemented to perform matrix-vector products in a block-wise fashion.

\begin{figure}
    \centering
    \includegraphics[width=\textwidth]{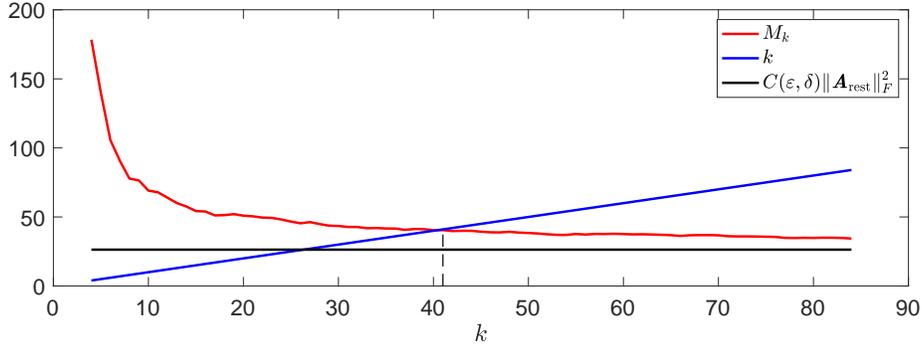}
    \caption{In this example we let $\bm{A} = \bm{U} \bm{\Lambda} \bm{U}^T \in \mathbb{R}^{1000 \times 1000}$ where $\bm{U}$ is a random orthogonal matrix and $\bm{\Lambda}$ is a diagonal matrix with $\bm{\Lambda}_{ii} = 1/i^{1.5}$. We run Algorithm \ref{alg:adap_hpp} with $\varepsilon = 0.01\tr(\bm{A})$, $\delta = 0.05$ and $\ell = 0$ to obtain $\bm{A}_{\text{rest}}$ defined in line~\ref{line:rsvdend}. The x-axis shows the number of matrix-vector products with $\bm{A}_{\text{rest}}$. The red line shows the evolution of the sequence $M_k$ defined in~\eqref{eq:M_k}, the blue line shows the linear line $k$ against $k$ and the black line is the number of matrix-vector products with $\bm{A}_{\text{rest}}$ to guarantee an error less than $\varepsilon$ with probability at least $1-\delta$. We stop the while loop in Algorithm~\ref{alg:adap_hpp_heuristic} once the red and blue line cross.}
    \label{fig:adaptive_M}
\end{figure}

\begin{algorithm}[H]
\caption{A-Hutch++}
\label{alg:adap_hpp_heuristic}
\textbf{input:} Symmetric $\bm{A} \in \mathbb{R}^{n \times n}$. Error tolerance $\varepsilon > 0$. Failure probability $\delta \in (0,1)$. Block-size $b$.\\
\textbf{output:} An approximation to $\tr(\bm{A}):\tr_{\mathsf{adap}}(\bm{A})$.
\begin{algorithmic}[1]
\State Perform lines~\ref{line:rsvdbegin}--\ref{line:rsvdend} in Algorithm~\ref{alg:adap_hpp} to get $\bm{Q}$, $\text{trest}_1$ and $\bm{A}_{\text{rest}}$.
    \State Initialize $\bm{\Psi}^{(0)} = \begin{bmatrix} \quad \end{bmatrix}$ and $\bm{C}^{(0)} = \begin{bmatrix} \quad \end{bmatrix}$.
    \State Initialize $M_0 = \infty$ and $k = 0$.
    \While{$M_k > k$}
        \State $k = k + b$
        \State $\alpha_{k} = \sup\left\{\alpha \in (0,1) : \frac{\gamma\left(\frac{k}{2},\alpha \frac{k}{2}\right)}{\Gamma\left(\frac{k}{2}\right)} \leq \delta\right\}$
        \State Generate a random matrix $\widehat{\bm{\Psi}}^{(k)}\in \mathbb{R}^{n \times b}$ and append $\bm{\Psi}^{(k)} = \begin{bmatrix} \bm{\Psi}^{(k-b)} & \widehat{\bm{\Psi}}^{(k)}\end{bmatrix}$.
        \State Compute $\widehat{\bm{C}}^{(k)}= \bm{A}_{\text{rest}}\widehat{\bm{\Psi}}^{(k)}$ and append $\bm{C}^{(k)} = \begin{bmatrix} \bm{C}^{(k-b)} & \widehat{\bm{C}}^{(k)} \end{bmatrix}$.
        \State Over-estimate $\|\bm{A}_{\text{rest}}\|_F^2$ with $\text{estFrob}_k = \frac{1}{k\alpha_k} \|\bm{C}^{(k)}\|_F^2$.
        \State Define $M_k = C(\varepsilon,\delta) \text{estFrob}_k$.
    \EndWhile
    \State \textbf{return} $\tr_{\mathsf{adap}}(\bm{A}) = \text{trest}_1 + \frac{1}{k} \text{tr}(\bm{\Psi}^{(k)T}\bm{C}^{(k)})$
\end{algorithmic}
\end{algorithm}

Due to the lack of independence between the Frobenius norm estimation and the stochastic trace estimation, the proof of Lemma~\ref{lemma:success1} does not extend to Algorithm~\ref{alg:adap_hpp_heuristic}. In turn, this algorithm does not come with the same type of success guarantee. However, as presented in Section \ref{section:synthetic_matrices} the empirical failure probabilities remain well below the prescribed failure probability.

\subsection{Numerical experiments}\label{sec:exp-ahpp}

All numerical experiments in this paper have been performed in Matlab, version R2020a; our implementation of Algorithm~\ref{alg:adap_hpp_heuristic} is available at \url{https://github.com/davpersson/A-Hutch-} together with the scripts to reproduce all figures and tables in this paper.

For a variety of matrices from~\cite{golubtrace,frommer,hpp,saibabaipsen}, we compare the newly proposed A-Hutch++ algorithm with Hutch++. In A-Hutch++ we fix $\delta = 0.05$ in all our experiments and we let $\varepsilon = \frac{|\tr(\bm{A})|}{2^p}$ for $p = 2,3,\ldots,10$, except in Figure~\ref{fig:TC} where we let $p = 3,4,\ldots,11$. The error of the estimate produced by A-Hutch++ implemented in a block-wise fashion is essentially identical to the unblocked version of A-Hutch++, i.e. $b = 1$, as long as the block-size is small compared to the number of required matrix-vector products. Therefore, for simplicity, we set the block-size to $b=1$ in all experiments. Furthermore, as discussed in Section \ref{sec:analysis_ahpp} we set $\ell = 0$ and omit the side condition on $m$.
For each considered matrix, for each value of $\varepsilon$, we first run Algorithm~\ref{alg:adap_hpp_heuristic} and count the number of matrix-vector products that have been used to obtain the estimate, then we run Algorithm~\ref{alg:hpp} with the same number of matrix-vector products.
For each value of $\varepsilon$ we repeat this 100 times and plot the average relative error on the y-axis and the average required matrix-vector products on the x-axis.
In each figure, the blue line is the average relative error from A-Hutch++, the red line is the average relative error from Hutch++, with the same number of matrix vector products, and the black dashed line is the $\varepsilon$ that was used as the input tolerance of A-Hutch++. For matrices with slow eigenvalue decay we have also included the average relative error from the Hutchinson estimator \eqref{eq:hutchinson}, see the green line in Figures~\ref{fig:algebraic_c=01},\ref{fig:algebraic_c=05},\ref{fig:TC}, and \ref{fig:tridiaginv}. The shaded blue area shows the $10^{\text{th}}$ to $90^{\text{th}}$ percentiles\footnote{We show the $90\%$ percentile because, if we did not reuse the matrix-vector products of the Frobenius norm estimation for the Hutchinson trace estimator, Lemma~\ref{lemma:success1} would ensure a failure probability of at most $2\delta=10\%$.} of the results from A-Hutch++, and the shaded red area shows the $10^{\text{th}}$ to $90^{\text{th}}$ percentiles of the results from Hutch++, see e.g. Figure \ref{fig:algebraic}. 

In the numerical experiments we observe that A-Hutch++ performs better compared to Hutch++ for matrices with slower singular value decay; see e.g. Figure \ref{fig:algebraic_c=01}, in which A-Hutch++ achieves an average relative error of 0.001827 using an average of 74.41 matrix-vector products ($6^{\text{th}}$ blue point in the figure). In comparison, Hutch++ achieves an average relative error of 0.001804 using an average of 237.7 matrix-vector products ($7^{\text{th}}$ red point in the figure). Hence, in these cases the adaptivity does improve the performance compared to Hutch++. For faster singular value decay the two algorithms perform similarly. However, in no case does Hutch++ perform noticeably better compared to A-Hutch++.


\subsubsection{Synthetic matrices}\label{section:synthetic_matrices}
We create matrices with algebraically decaying singular values as in~\cite{hpp}, i.e. $\bm{A} = \bm{U}\bm{\Lambda} \bm{U}^T \in \mathbb{R}^{5000 \times 5000}$ where $\bm{U}$ is a random orthogonal matrix and $\bm{\Lambda}$ is a diagonal matrix with $\bm{\Lambda}_{ii} = 1/i^c $ for $i = 1, \ldots, 5000$, for a parameter $c \in \{0.1, 0.5, 1, 3\}$. The results are shown in Figure~\ref{fig:algebraic}.

Furthermore, using these example matrices we also estimated the failure probability of A-Hutch++. Table \ref{table:failure_probabilities} demonstrates the empirical failure probabilities from 100000 repeats of A-Hutch++ for different input pairs $(\varepsilon,\delta)$. In all cases the empirical failure probabilities remain well below the prescribed failure probability. 

In addition, to demonstrate that A-Hutch++ allocates more matrix-vector products to the Hutchinson estimator for matrices with slow eigenvalue decay and vice versa for matrices with fast eigenvalue decay, we also include a table displaying the distribution of the matrix-vector products between the two phases. See Table~\ref{table:matvecs_dist}.

\begin{table}[htp]
    \centering
    \begin{subtable}[h]{0.45\textwidth}
        \begin{tabular}{|l||*{3}{c|}}\hline
            \backslashbox[21mm]{$\varepsilon$}{$\delta$}
            &\makebox[3em]{$0.1$}&\makebox[3em]{$0.05$}&\makebox[3em]{$0.01$}\\\hline\hline
            $0.1\tr(\bm{A})$ &0&0&0\\\hline
            $0.01\tr(\bm{A})$ &0.00285&0.00076&0.00005\\\hline
            $0.005\tr(\bm{A})$ &0.00686&0.00244&0.00015\\\hline
        \end{tabular}
        \caption{$c=0.1$}
    \end{subtable}
    
    \bigskip
    \begin{subtable}[h]{0.45\textwidth}
        \centering
        \begin{tabular}{|l||*{3}{c|}}\hline
            \backslashbox[21mm]{$\varepsilon$}{$\delta$}
            &\makebox[3em]{$0.1$}&\makebox[3em]{$0.05$}&\makebox[3em]{$0.01$}\\\hline\hline
            $0.1\tr(\bm{A})$ &0&0&0\\\hline
            $0.01\tr(\bm{A})$ &0.00484&0.00126&0.00010\\\hline
            $0.005\tr(\bm{A})$ &0.00855&0.00331&0.00032\\\hline
        \end{tabular}
        \caption{$c=0.5$}
    \end{subtable}
    
    \bigskip
    \begin{subtable}[h]{0.45\textwidth}
        \centering
        \begin{tabular}{|l||*{3}{c|}}\hline
            \backslashbox[21mm]{$\varepsilon$}{$\delta$}
            &\makebox[3em]{$0.1$}&\makebox[3em]{$0.05$}&\makebox[3em]{$0.01$}\\\hline\hline
            $0.1\tr(\bm{A})$ &0.00026&0.00002&0\\\hline
            $0.01\tr(\bm{A})$ &0.00607&0.00186&0.00018\\\hline
            $0.005\tr(\bm{A})$ &0.00804&0.00250&0.00030\\\hline
        \end{tabular}
        \caption{$c=1$}
    \end{subtable}
    
    \bigskip
    \begin{subtable}[h]{0.45\textwidth}
        \centering
        \begin{tabular}{|l||*{3}{c|}}\hline
            \backslashbox[21mm]{$\varepsilon$}{$\delta$}
            &\makebox[3em]{$0.1$}&\makebox[3em]{$0.05$}&\makebox[3em]{$0.01$}\\\hline\hline
            $0.1\tr(\bm{A})$ &0&0&0\\\hline
            $0.01\tr(\bm{A})$ &0.00002&0&0\\\hline
            $0.005\tr(\bm{A})$ &0.00006&0&0\\\hline
        \end{tabular}
        \caption{$c=3$}
    \end{subtable}
    \caption{Empirical failure probabilities from 100000 repeats of applying A-Hutch++ on the synthetic matrices described in Section \ref{section:synthetic_matrices}.}
    \label{table:failure_probabilities}
\end{table}

\begin{figure}
\begin{subfigure}{.5\textwidth}
  \centering
  \includegraphics[width=\linewidth]{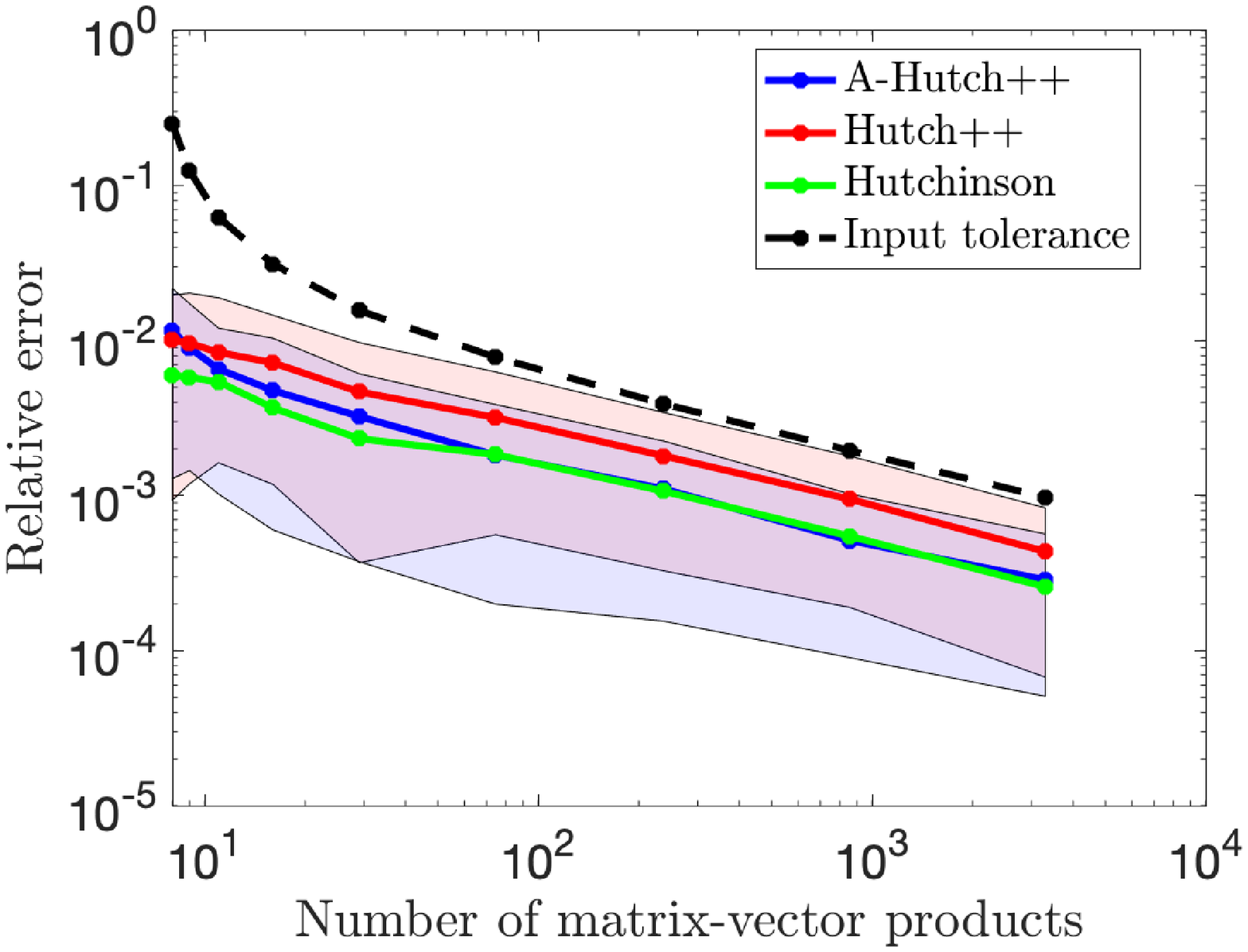}  
  \caption{$c = 0.1$}
  \label{fig:algebraic_c=01}
\end{subfigure}
\begin{subfigure}{.5\textwidth}
  \centering
  \includegraphics[width=\linewidth]{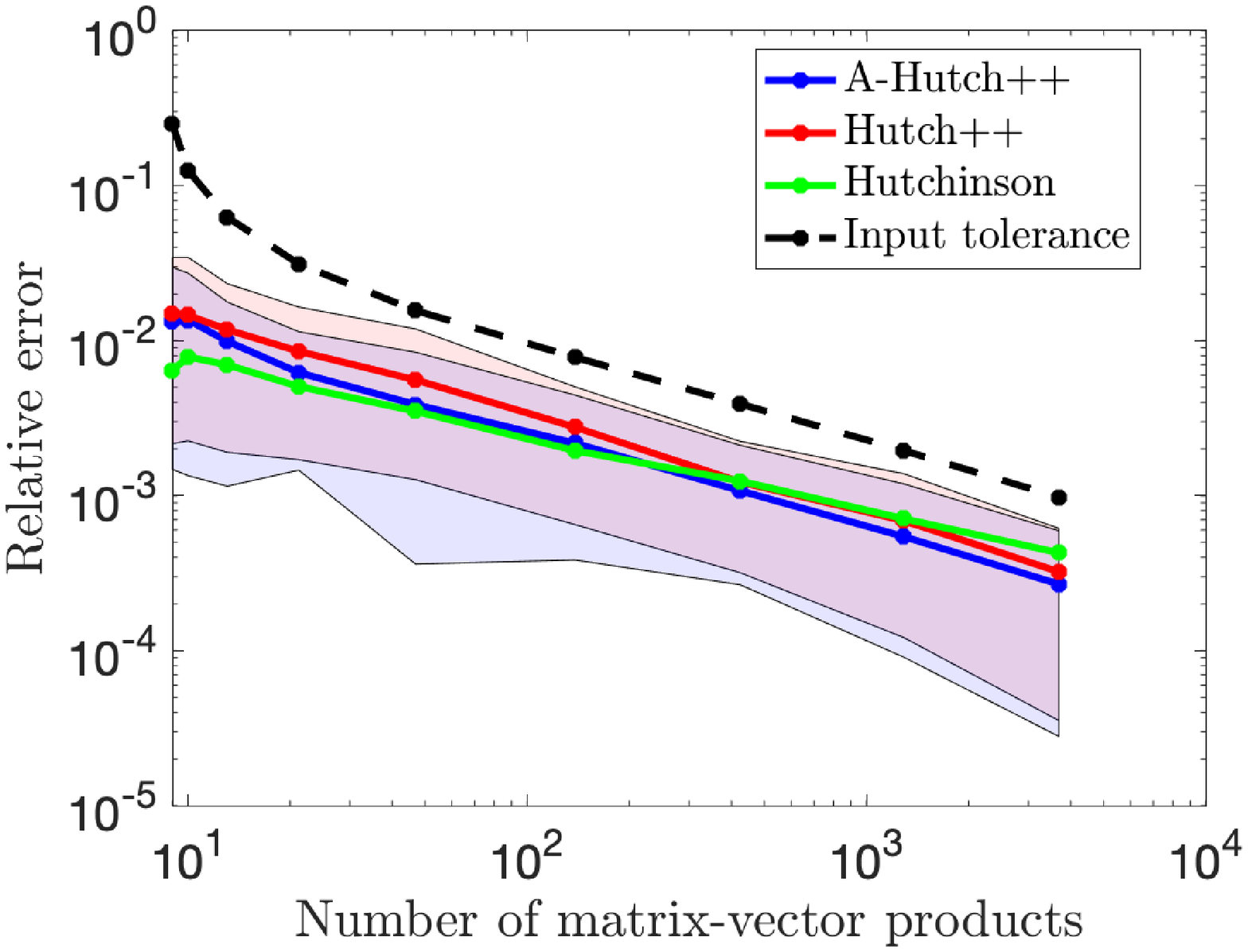}  
  \caption{$c=0.5$}
  \label{fig:algebraic_c=05}
\end{subfigure}

\begin{subfigure}{.5\textwidth}
  \centering
  \includegraphics[width=\linewidth]{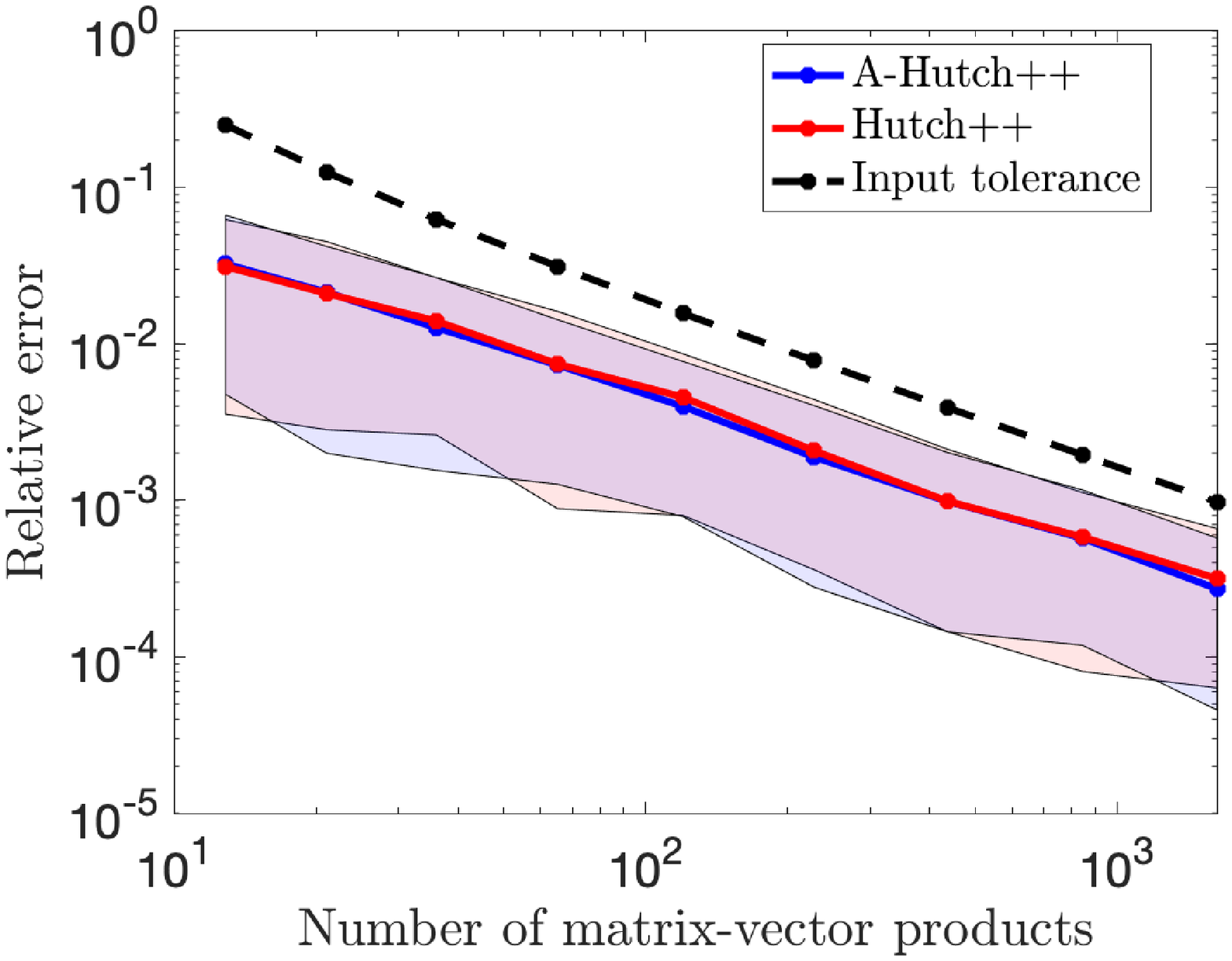}  
  \caption{$c=1$}
  \label{fig:algebraic_c=1}
\end{subfigure}
\begin{subfigure}{.5\textwidth}
  \centering
  \includegraphics[width=\linewidth]{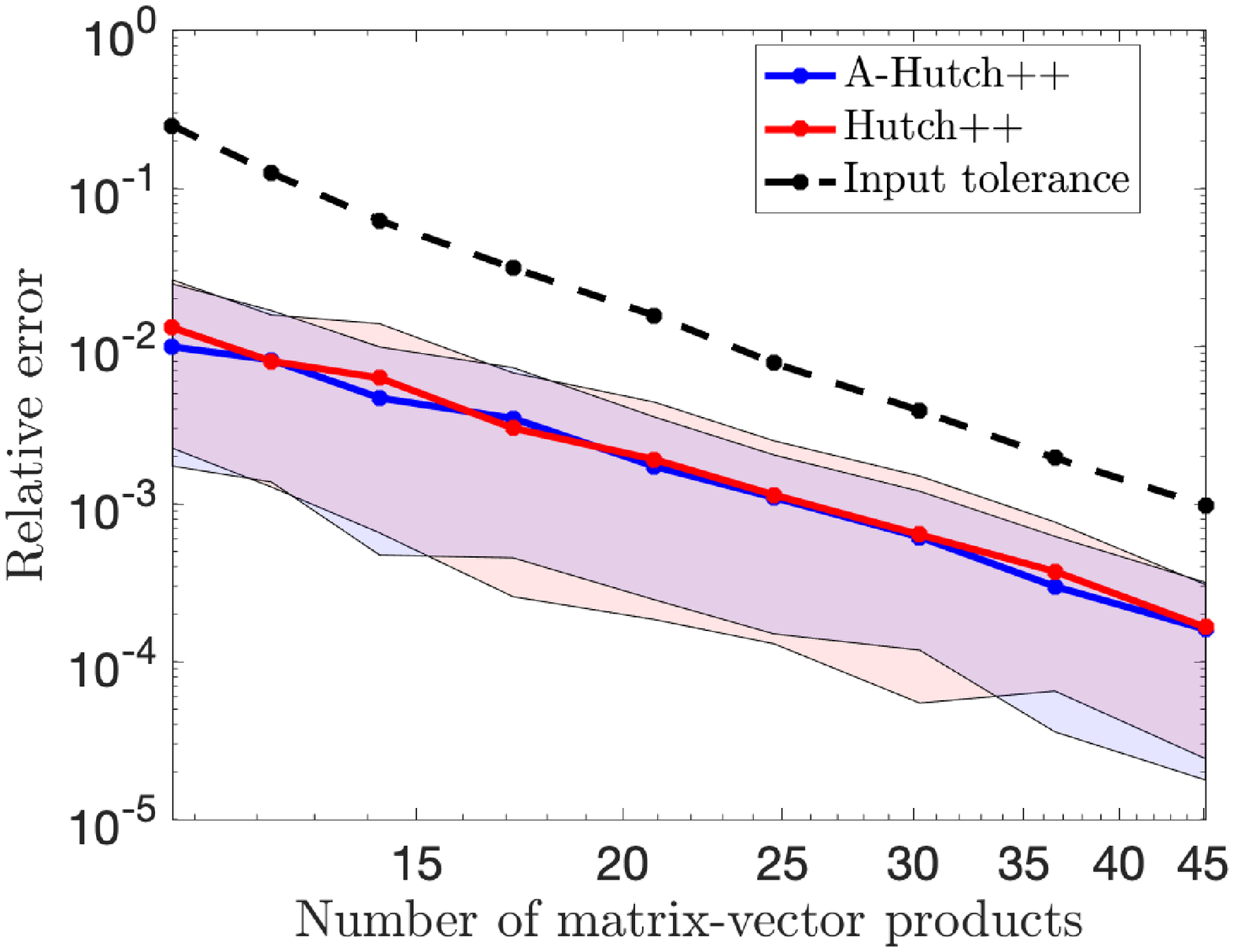}  
  \caption{$c=3$}
  \label{fig:algebraic_c=3}
\end{subfigure}
\caption{Comparison of A-Hutch++ and Hutch++ for the estimation of the trace of the synthetic matrices with algebraic decay from Section~\ref{section:synthetic_matrices}.}
\label{fig:algebraic}
\end{figure}

\begin{table}
\centering
    \begin{subtable}[c]{\textwidth}
    \centering
    \begin{tabular}{|l|l|l|l|l|}
    \hline
    $p$ & Total & Low rank approx. & Hutchinson est. & Ratio \\
    \hline \hline
    2  & 8.00    & 6.00 & 2.00    & 0.25 \\ \hline
    3  & 9.00    & 6.00 & 3.00    & 0.33 \\ \hline
    4  & 11.00   & 6.00 & 5.00    & 0.45 \\ \hline
    5  & 16.00   & 6.00 & 10.00   & 0.63 \\ \hline
    6  & 29.04   & 6.00 & 23.04   & 0.79 \\ \hline
    7  & 74.41   & 6.00 & 68.41   & 0.92 \\ \hline
    8  & 237.66  & 6.00 & 231.66  & 0.97 \\ \hline
    9  & 858.13  & 6.00 & 852.13  & 0.99 \\ \hline
    10 & 3302.76 & 6.00 & 3296.76 & 1.00 \\ \hline
    \end{tabular}
    \caption{$c=0.1$}
    \vspace{3mm}  
    \end{subtable}
    \\
    \begin{subtable}[c]{\textwidth}
    \centering
    \begin{tabular}{|l|l|l|l|l|}
    \hline
    $p$ & Total & Low rank approx. & Hutchinson est. & Ratio \\
    \hline \hline
    2  & 9.00    & 6.00   & 3.00    & 0.33 \\ \hline
    3  & 10.01   & 6.00   & 4.01    & 0.40 \\ \hline
    4  & 13.06   & 6.00   & 7.06    & 0.54 \\ \hline
    5  & 21.21   & 6.00   & 15.21   & 0.72 \\ \hline
    6  & 46.94   & 6.02   & 40.92   & 0.87 \\ \hline
    7  & 138.24  & 10.14  & 128.10  & 0.93 \\ \hline
    8  & 424.31  & 49.18  & 375.13  & 0.88 \\ \hline
    9  & 1287.60 & 206.96 & 1080.64 & 0.84 \\ \hline
    10 & 3688.39 & 914.18 & 2774.21 & 0.75 \\ \hline
    \end{tabular}
    \caption{$c=0.5$}
    \vspace{3mm} 
    \end{subtable}
    \\
    \begin{subtable}[c]{\textwidth}
    \centering
    \begin{tabular}{|l|l|l|l|l|}
    \hline
    $p$ & Total & Low rank approx. & Hutchinson est. & Ratio \\
    \hline \hline
   2  & 12.86   & 6.16   & 6.70   & 0.52 \\ \hline
3  & 21.07   & 8.86   & 12.21  & 0.58 \\ \hline
4  & 36.02   & 15.08  & 20.94  & 0.58 \\ \hline
5  & 65.15   & 27.68  & 37.47  & 0.58 \\ \hline
6  & 120.04  & 52.84  & 67.20  & 0.56 \\ \hline
7  & 228.02  & 101.18 & 126.84 & 0.56 \\ \hline
8  & 436.75  & 199.14 & 237.61 & 0.54 \\ \hline
9  & 843.98  & 396.32 & 447.66 & 0.53 \\ \hline
10 & 1630.29 & 793.36 & 836.93 & 0.51 \\ \hline
    \end{tabular}
    \caption{$c=1$}
    \vspace{3mm} 
    \end{subtable}
    \end{table}
    \pagebreak

\begin{table}
\ContinuedFloat
    \begin{subtable}[c]{\textwidth}
    \centering
        \begin{tabular}{|l|l|l|l|l|}
    \hline
    $p$ & Total & Low rank approx. & Hutchinson est. & Ratio \\
    \hline \hline
   2  & 10.66 & 8.20  & 2.46  & 0.23 \\ \hline
3  & 12.24 & 8.88  & 3.36  & 0.27 \\ \hline
4  & 14.24 & 10.76 & 3.48  & 0.24 \\ \hline
5  & 17.16 & 12.44 & 4.72  & 0.28 \\ \hline
6  & 20.91 & 15.22 & 5.69  & 0.27 \\ \hline
7  & 24.70 & 18.28 & 6.42  & 0.26 \\ \hline
8  & 30.28 & 22.50 & 7.78  & 0.26 \\ \hline
9  & 36.57 & 27.68 & 8.89  & 0.24 \\ \hline
10 & 45.14 & 34.50 & 10.64 & 0.24 \\ \hline
    \end{tabular}
    \caption{$c=3$}
    \vspace{3mm} 
    \end{subtable}
    \caption{The average distribution of matrix-vector products between the low rank approximation phase and stochastic trace esimation phase of A-Hutch++ applied on the synthetic matrices with algebraic decay and input tolerance $\varepsilon = 2^{-p} \tr(\bm{A})$ for $p = 2,3,\ldots,10$. A-Hutch++ requires at least 6 matrix-vector products to detect a minimum of the function $\tilde{m}(r)$ in \eqref{eq:tildeN}.}\label{table:matvecs_dist}
\end{table}

\subsubsection{Triangle counting}\label{sec:triangles}
For an undirected graph with adjacency matrix $\bm{B}$, the number of triangles in the graph is equal to $\frac{1}{6} \tr(\bm{B}^3)$; counting triangles arises for instance in data mining applications~\cite{avron}. We apply A-Hutch++ and Hutch++ to $\bm{A} = \bm{B}^3$, where $\bm{B}$ is the adjacency matrix of the following graphs:
\begin{itemize}
    \item a Wikipedia vote network\footnote{Accessed from \url{https://snap.stanford.edu/data/wiki-Vote.html}} of size $7115$ ($\tr(\bm{B}^3) = 3650334$);
    \item an arXiv collaboration network\footnote{Accessed from \url{https://snap.stanford.edu/data/ca-GrQc.html}} of size $5242$ ($\tr(\bm{B}^3) = 289560$).
    
\end{itemize}
Note that one matrix-vector product with $\bm{A}$ corresponds to three matrix-vector products with $\bm{B}$. 
The numerical results are shown in Figure~\ref{fig:triangle}.

\begin{figure}[H]
\begin{subfigure}{.5\textwidth}
  \centering
  \includegraphics[width=\linewidth]{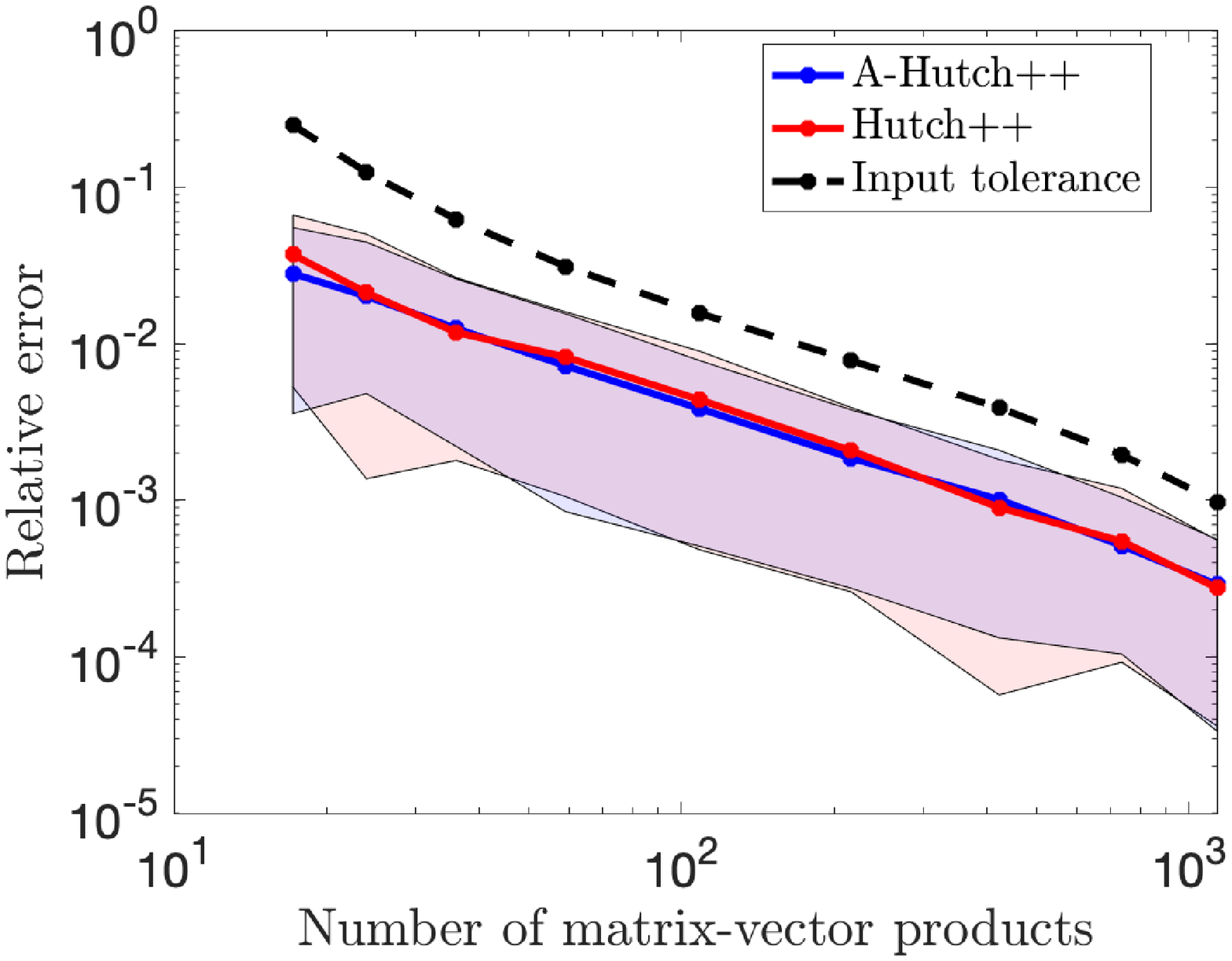}  
  \caption{Wikipedia vote network}
  \label{fig:wiki}
\end{subfigure}
\begin{subfigure}{.5\textwidth}
  \centering
  \includegraphics[width=\linewidth]{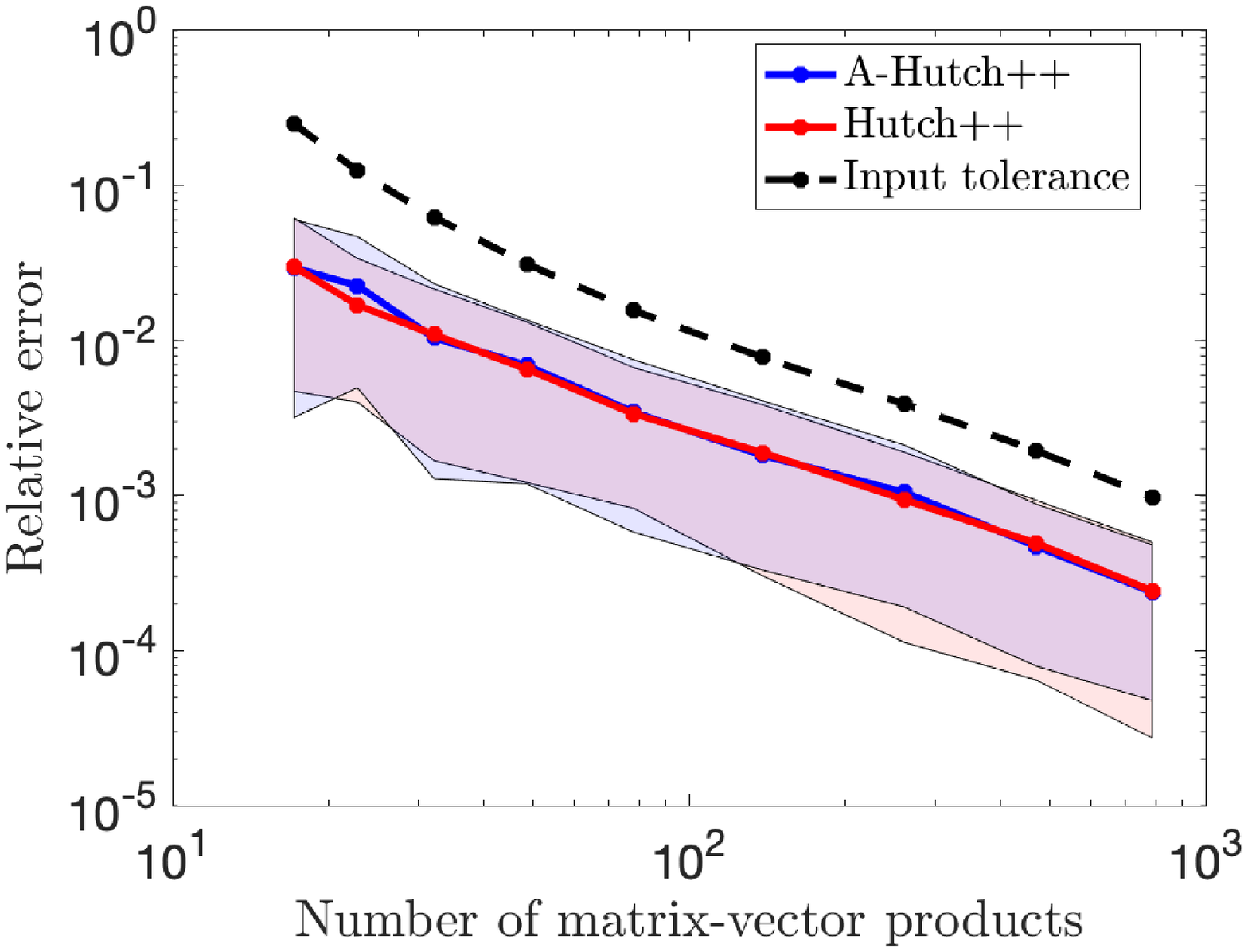}  
  \caption{Arxiv GR-QC}
  \label{fig:arxiv}
\end{subfigure}
\caption{Comparison of A-Hutch++ and Hutch++ for the triangle counting examples from Section~\ref{sec:triangles}.}
\label{fig:triangle}
\end{figure}

\subsubsection{Estrada index}\label{section:estrada}
For an undirected graph with adjacency matrix $\bm{B}$, the Estrada index is defined as $\tr(\exp(\bm{B}))$ and its applications include measuring the degree of protein protein folding~\cite{estrada} and network analysis~\cite{estrada2010}. As in~\cite{hpp}, we estimate the Estrada index of Roget’s Thesaurus semantic graph adjacency matrix\footnote{Accessed from \url{http://vlado.fmf.uni-lj.si/pub/networks/data/}}. 
We approximate matrix-vector products with $\bm{A} = \exp(\bm{B})$ using 30 iterations of the Lanczos method~\cite[Chapter 13.2]{higham2008functions}, after which the error from the approximated matrix-vector product is negligible. The results are shown in Figure~\ref{fig:estrada}. 

\begin{figure}[H]
  \centering
  \includegraphics[width=.5\linewidth]{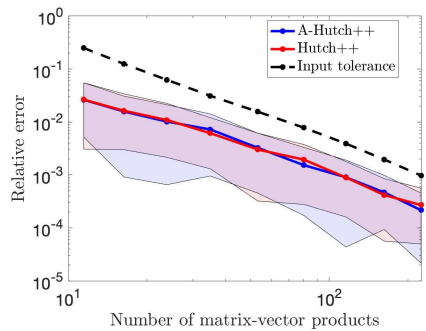}  
  \caption{Comparison of A-Hutch++ and Hutch++ for the estimation of the Estrada index of the matrix from Section~\ref{section:estrada}.}
\label{fig:estrada}
\end{figure}

\subsubsection{Log-determinant}\label{section:logdet}

The computation of the log-determinant of a symmetric positive definite matrix, which arises for instance in statistical learning~\cite{affandi} and Markov random fields models~\cite{wainwright2006log}, can be addressed by trace estimation exploiting the relation
\begin{equation*}
    \log\det(\bm{B}) = \tr(\log(\bm{B})).
\end{equation*}
In our setting we apply A-Hutch++ and Hutch++ to $\bm{A} =\log(\bm{B})$ for the following symmetric positive definite matrices $\bm{B}$:
\begin{itemize}
    \item  $\bm{B} = \bm{I} + \sum\limits_{j=1}^{40}\frac{10}{j^2}\bm{x}_j \bm{x}_j^T + \sum\limits_{j=41}^{300} \frac{1}{j^2} \bm{x}_j \bm{x}_j^T$ where $\bm{x}_1,\cdots,\bm{x}_{300} \in \mathbb{R}^{5000}$ are generated in Matlab using \texttt{sprandn(5000,1,0.025)}. This example comes from~\cite{saibabaipsen,sorensen}. 
    $\bm{B}$ has an eigenvalue gap at index $40$. Matrix-vector products with $\bm{A}=\log(\bm{B})$ are approximated using $25$ iterations of Lanczos method. 
    \item  $\bm{B}$ is the Thermomech TC matrix\footnote{Accessed from \url{https://sparse.tamu.edu/Botonakis/thermomech\_TC}} from the SuiteSparse Matrix Collection~\cite{sparse}. Matrix-vector products with $\bm{A}=\log(\bm{B})$ are approximated using $35$ iterations of Lanczos method. 
\end{itemize}
The numerical results are shown in Figure~\ref{fig:logdet}. 

\begin{figure}
\begin{subfigure}{.5\textwidth}
  \centering
  \includegraphics[width=\linewidth]{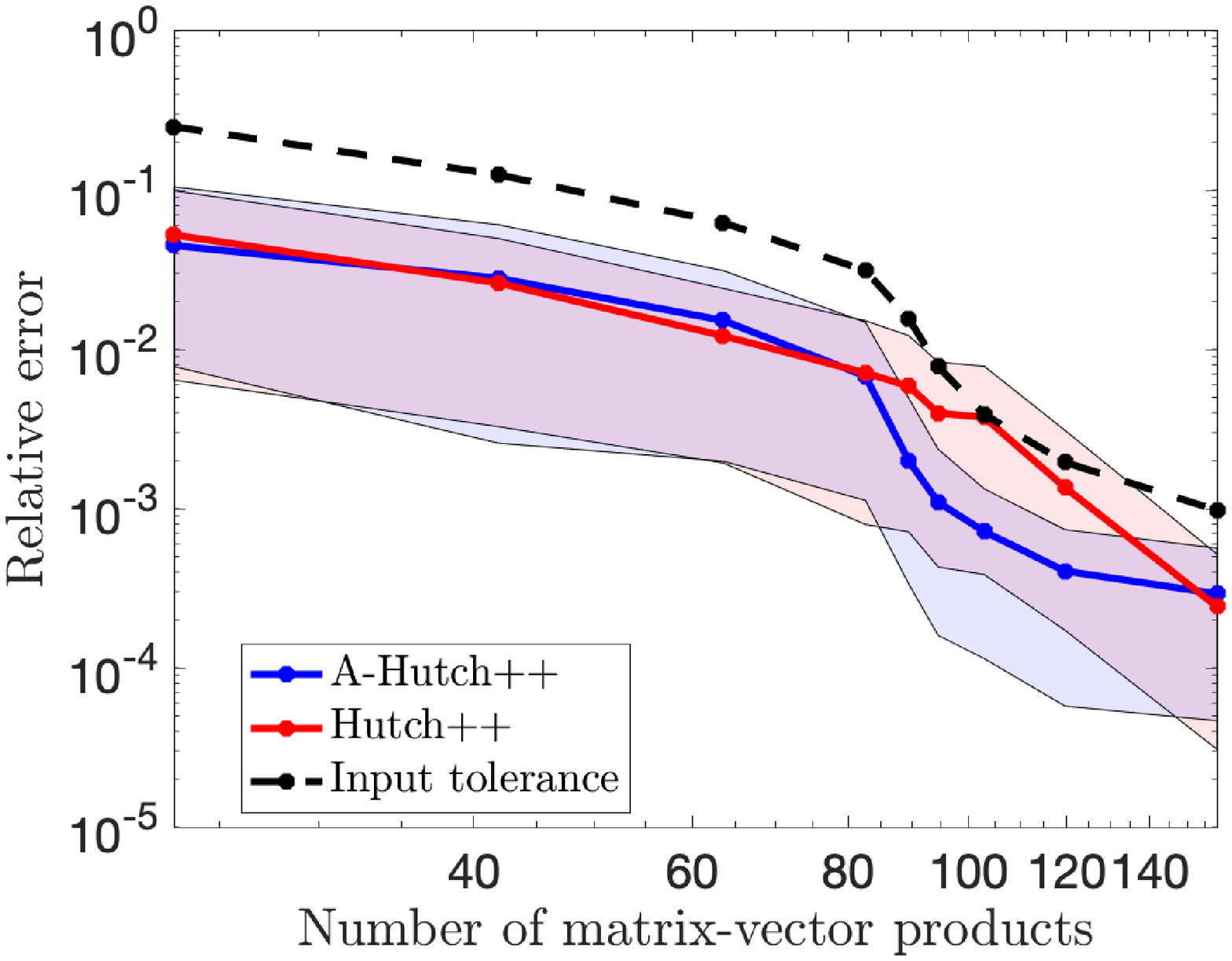}  
  \caption{Estimating the log-determinant of the matrix from \cite{saibabaipsen}.}
  \label{fig:saibaba_ipsen}
\end{subfigure}
\begin{subfigure}{.5\textwidth}
  \centering
  \includegraphics[width=\linewidth]{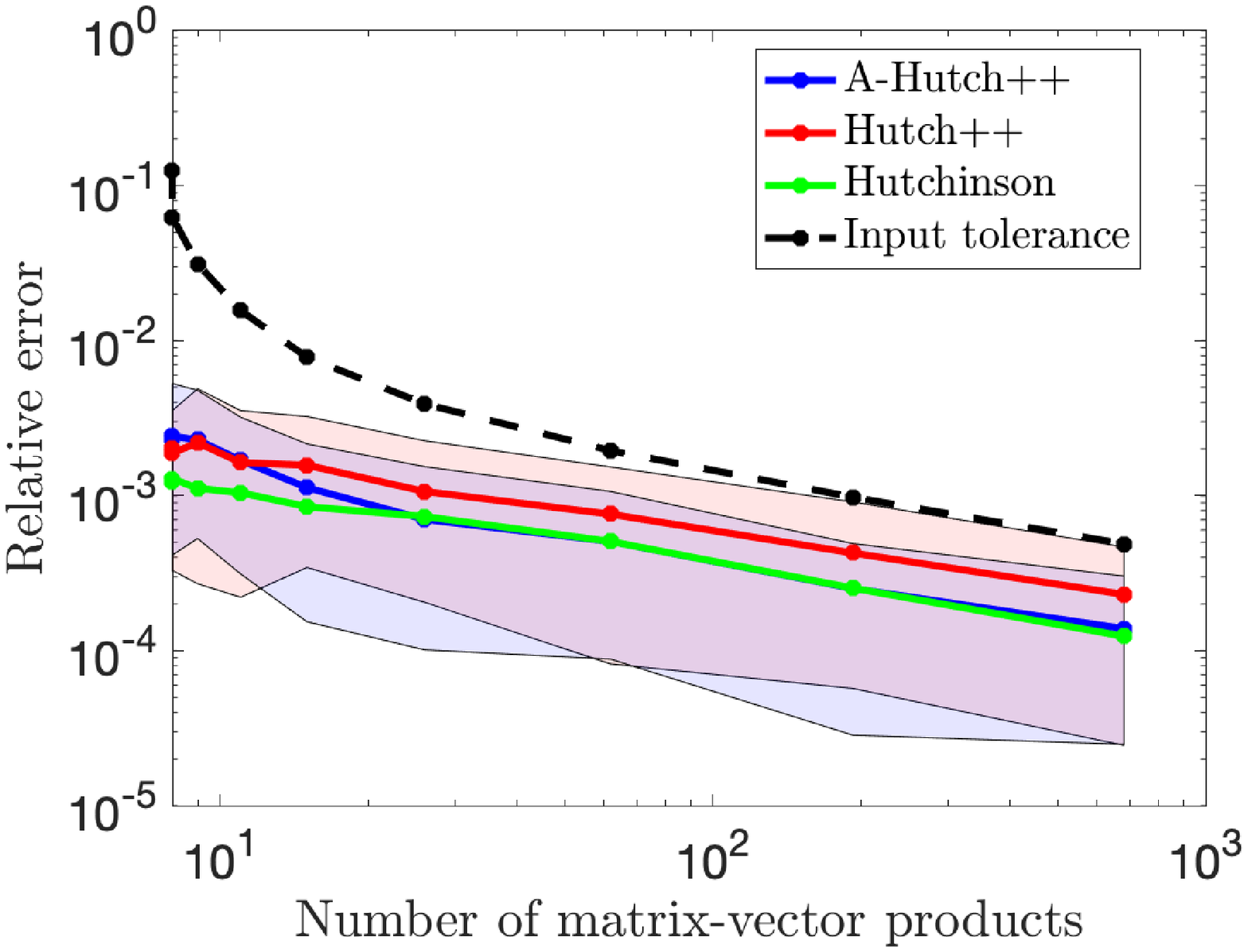}  
  \caption{Estimating the log-determinant of the matrix Thermomech TC.}
  \label{fig:TC}
\end{subfigure}
\caption{Comparison of A-Hutch++ and Hutch++ for the log determinant estimation of the matrices from Section~\ref{section:logdet}.}
\label{fig:logdet}
\end{figure}

\subsubsection{Trace of inverses}\label{section:inverse}

We consider $\bm{A} = \bm{B}^{-1}$ for the following choices of $\bm{B}$:
\begin{itemize}
    \item $\bm{B} = \text{tridiag}(-1,4,-1)$ is a $10000 \times 10000$ tridiagonal matrix with $4$ along the diagonal and $-1$ along the upper and lower subdiagonal (taken from~\cite{frommer});
    \item $\bm{B}$ a block tridiagonal matrix of size $k^2 \times k^2$ generated from discretizing Poisson's equation with the 5-point operator on a $k \times k$ mesh, with $k = 100$ (taken from~\cite{golubtrace}).
\end{itemize}
Matrix-vector products with $\bm{A} = \bm{B}^{-1}$ are computed using backslash in Matlab. The results are shown in Figure~\ref{fig:inv}.

\begin{figure}
\begin{subfigure}{.5\textwidth}
  \centering
  \includegraphics[width=\linewidth]{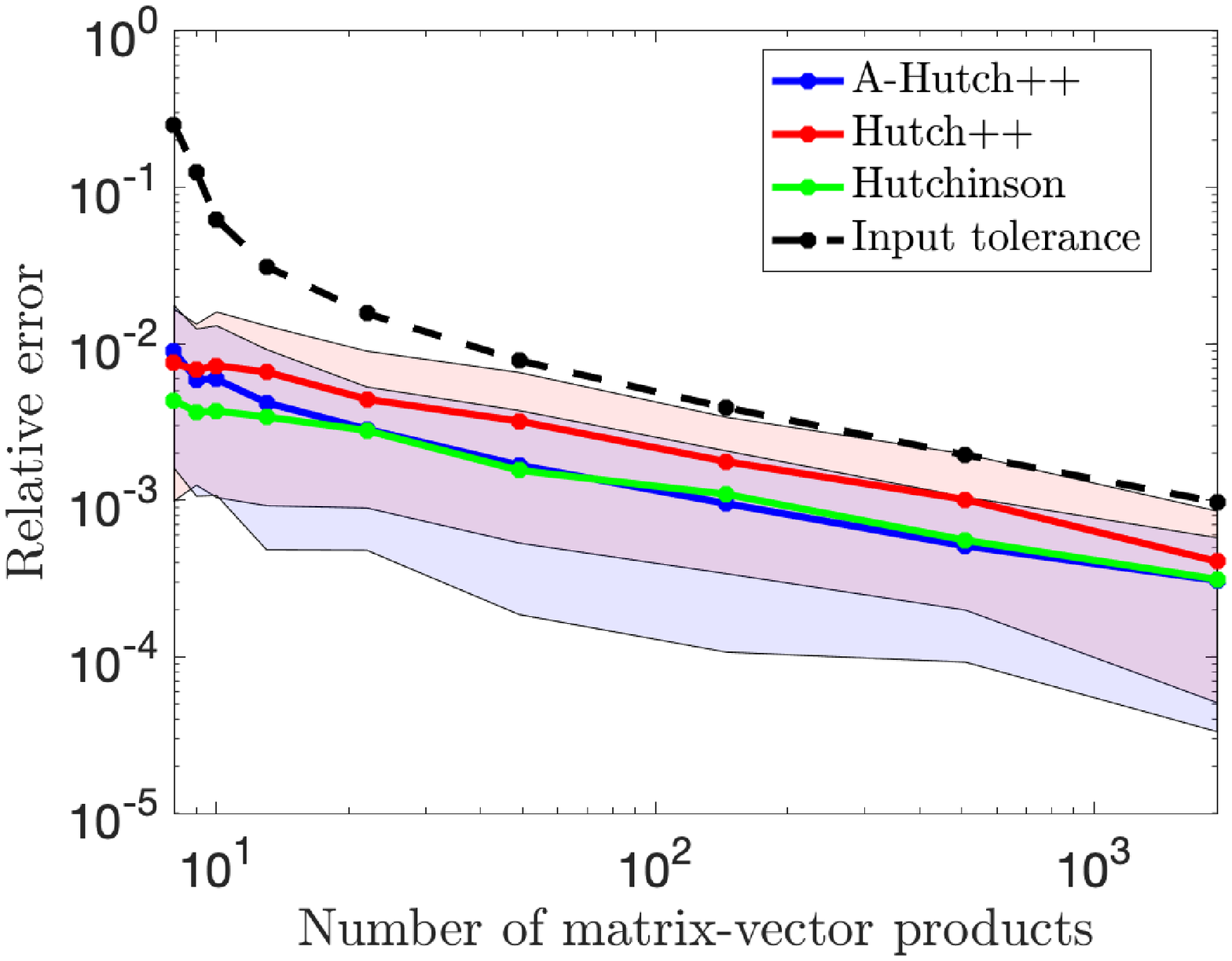}  
  \caption{Inverse of $\text{tridiag}(-1,4,-1)$.}
  \label{fig:tridiaginv}
\end{subfigure}
\begin{subfigure}{.5\textwidth}
  \centering
  \includegraphics[width=\linewidth]{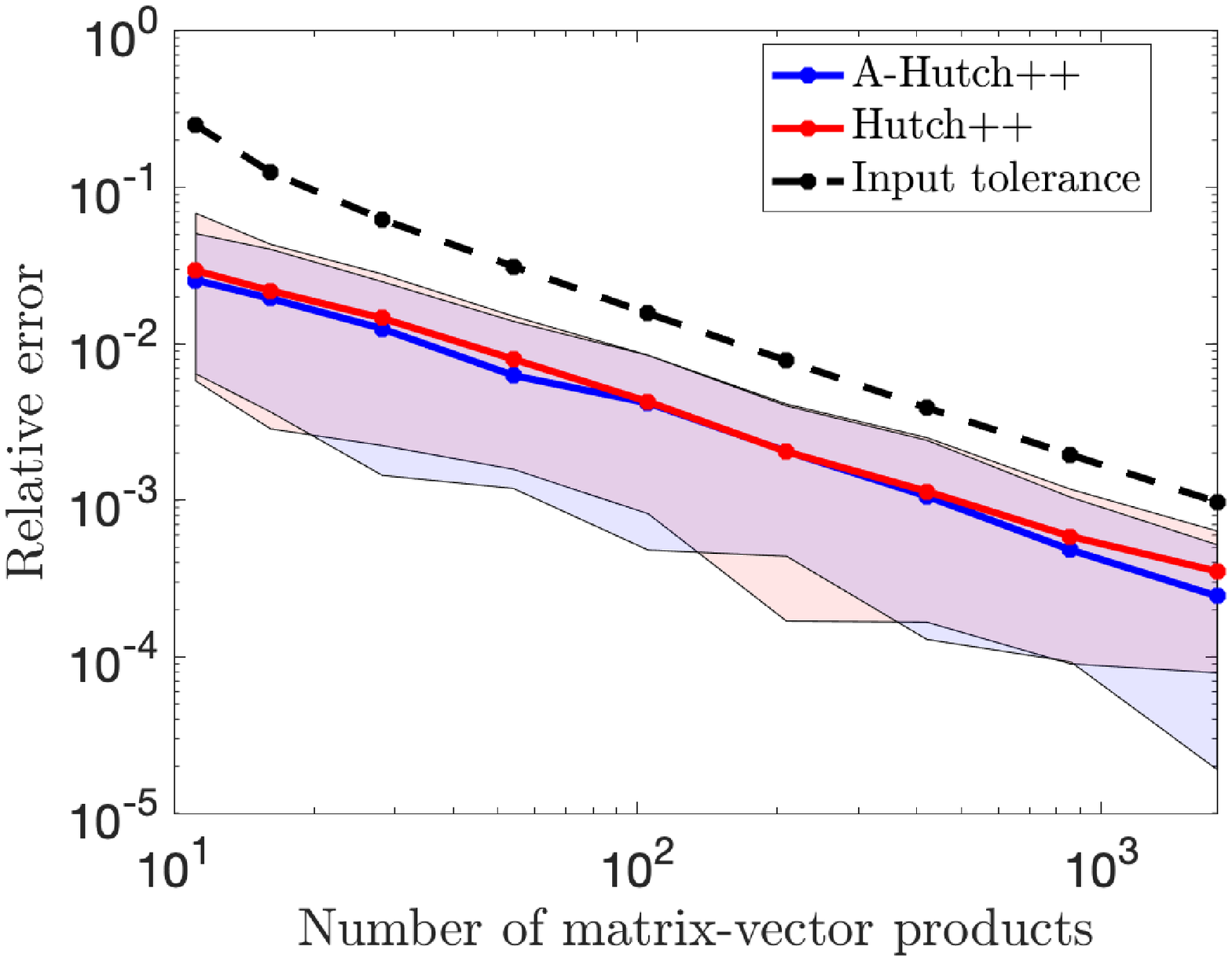}  
  \caption{Inverse of the matrix generated from discretizing Poisson's equation.}
  \label{fig:poissoninv}
\end{subfigure}
\caption{Comparison of A-Hutch++ and Hutch++ for the estimation of the trace of the inverse of the matrices described in Section~\ref{section:inverse}.}
\label{fig:inv}
\end{figure}

\section{Nyström++}

As explained in the introduction, Hutch++ requires at least two passes over the matrix $\bm{A}$.
In \cite{hpp}, Algorithm~\ref{alg:sp_hpp} was presented, and its analysis was improved in \cite{nahutchpp}. It requires only one pass over the input matrix, when computing the matrix vector products in line~\ref{line:blockmatvec}, and we thus call it \textit{Single Pass Hutch++}.
\begin{algorithm}
\caption{Single Pass Hutch++}
\label{alg:sp_hpp}
\textbf{input:} Symmetric positive semi-definite $\bm{A} \in \mathbb{R}^{n \times n}$. Number of matrix-vector products $m \in \mathbb{N}$.\\
\textbf{output:} An approximation to $\tr(\bm{A}): \tr_{m}^{\mathsf{sph++}}(\bm{A})$
\begin{algorithmic}[1]
    \State Fix positive constants $c_1,c_2$ and $c_3$ such that $c_1 < c_2$ and $c_1 + c_2 + c_3 = 1$
    \State Sample $\bm{\Omega} \in \mathbb{R}^{d \times c_1 m}, \bm{\Psi} \in \mathbb{R}^{d \times c_2 m}, \bm{\Phi} \in \mathbb{R}^{d \times c_3m}$ with i.i.d. $N(0,1)$ or Rademacher entries
    \State Compute $\begin{bmatrix} \bm{X} & \bm{Y} & \bm{Z} \end{bmatrix}  = \bm{A}\begin{bmatrix} \bm{\Omega} & \bm{\Psi} & \bm{\Phi} \end{bmatrix}$ \label{line:blockmatvec}
    \State \textbf{return} $\tr_m^{\mathsf{sph++}}(\bm{A}) = \tr((\bm{\Omega}^T \bm{Y})^{\dagger} (\bm{X}^T \bm{Y})) + \frac{1}{c_3m} (\tr(\bm{\Phi}^T \bm{Z})- \tr(\bm{\Phi}^T \bm{Y}(\bm{\Omega}^T\bm{Y})^{\dagger}\bm{X}^T \bm{\Phi}))$
\end{algorithmic}
\end{algorithm}
It also fits the streaming model because an update $\bm{A} + \bm{E}$ of the input matrix trivially translated into an update of the matrix-vector products, without having to revisit $\bm{A}$. It is similar to Hutch++ since it consists of a randomized low rank approximation phase and a stochastic trace estimation phase. The low rank approximation phase is performed by computing the low rank approximation $\bm{A\Psi}(\bm{\Omega}^T \bm{A} \bm{\Psi})^{\dagger} (\bm{A}\bm{\Omega})^T = \bm{Y}(\bm{\Omega}^T \bm{Y})^{\dagger} \bm{X}^T$, where $\bm{X},\bm{Y}$ and $\bm{Z}$ are as in line~\ref{line:blockmatvec} of Single Pass Hutch++. The trace of the low rank approximation equals $\tr((\bm{\Omega}^T \bm{Y})^{\dagger} (\bm{X}^T \bm{Y}))$ via the cyclic property of the trace. In the stochastic trace estimation phase the trace of $\bm{A}-\bm{Y}(\bm{\Omega}^T \bm{Y})^{\dagger} \bm{X}^T$ is estimated, which is done by the stochastic trace estimator~\eqref{eq:hutchinson}.
Single Pass Hutch++ satisfies similar guarantees as Hutch++, but is observed to produce a less accurate trace estimate than Hutch++ with the same number of matrix-vector products. More formally, the following result was proved.
\begin{theorem}[{\cite[Theorem 1.1]{nahutchpp}}]
If Single Pass Hutch++ is implemented with $m = O\left(\varepsilon^{-1}\sqrt{\log(\delta^{-1})} + \log(\delta^{-1})\right)$ matrix-vector products then 
\begin{equation*}
    \left|\tr_m^{\mathsf{sph++}}(\bm{A})-\tr(\bm{A})\right| \leq \varepsilon \tr(\bm{A}).
\end{equation*}
holds with probability at least $1-\delta$.
\end{theorem}
On the other hand, the numerical experiments in \cite{nahutchpp} demonstrated that due to the single pass property, which allows for performing matrix-vector products in parallel, Single Pass Hutch++ outperforms Hutch++ in terms of wall-clock time.\footnote{One needs to be careful how to implement the low-rank approximation in Single Pass Hutch++, since it is prone to numerical instabilities due to the pseudoinverse of $\bm{\Omega}^T \bm{Y}$. In our implementation we follow the suggestion given in \cite[Section 5.1]{nakatsukasa2020fast}. We compute a thin QR-decomposition of $(\bm{\Omega}^T \bm{Y})^T = \bm{QR}$ and let $\bm{S} = \bm{Y}\bm{Q}$ and $\bm{Z} = \bm{X}\bm{R}^{-1}$. Then $\bm{Y}(\bm{\Omega}^T \bm{Y})^{\dagger} \bm{X}^T = \bm{S}\bm{Z}^T.$}

For symmetric positive semi-definite $\bm{A}$ one can obtain a version of Hutch++ by utilizing the Nyström approximation $\bm{A} \approx \bm{A\Omega}(\bm{\Omega}^T \bm{A} \bm{\Omega})^{\dagger} \bm{\Omega}^T \bm{A}$~\cite{gittensmahoney} instead. We call this algorithm Nyström++, see Algorithm~\ref{alg:npp}. The idea of using the Nyström approximation in the context of trace estimation had previously been presented in~\cite[Section 4]{linlin} in a broader context, but no analysis was presented. A version of Hutch++ using a similar low-rank approximation was also mentioned in \cite{ramhutch++}. Furthermore, Nyström++ also fits the streaming model. Another possible advantage of Nyström++ over Hutch++ is that while the Nyström approximation is less accurate than the randomized SVD, one can spend more matrix-vector products for both attaining a low-rank approximation of $\bm{A}$ and on estimating the trace of $\bm{A}-\bm{A\Omega}(\bm{\Omega}^T \bm{A\Omega})^{\dagger} \bm{\Omega}^T\bm{A}$. 
\begin{algorithm}[H]
\caption{Nyström++}
\label{alg:npp}
\textbf{input:} Symmetric positive semi-definite $\bm{A} \in \mathbb{R}^{n \times n}$. Number of matrix-vector products $m \in \mathbb{N}$ (multiple of 2).\\
\textbf{output:} An approximation to $\tr(\bm{A}): \tr_{m}^{\mathsf{n++}}(\bm{A})$.
\begin{algorithmic}[1]
    \State Sample $\bm{\Omega} \in \mathbb{R}^{n \times \frac{m}{2}}, \bm{\Phi} \in \mathbb{R}^{n \times \frac{m}{2}}$ with i.i.d. $N(0,1)$ entries.
    \State Compute $\begin{bmatrix} \bm{X} & \bm{Y} \end{bmatrix}  = \bm{A}\begin{bmatrix} \bm{\Omega} & \bm{\Phi}\end{bmatrix}$.
    \State \textbf{return} $\tr_m^{\mathsf{n++}}(\bm{A}) = \text{tr}((\bm{\Omega}^T \bm{X})^{\dagger} (\bm{X}^T \bm{X})) + \frac{2}{m} (\text{tr}(\bm{\Phi}^T \bm{Y})- \tr(\bm{\Phi}^T \bm{X}(\bm{\Omega}^T\bm{X})^{\dagger}\bm{X}^T \bm{\Phi}))$ \label{line:return_npp}
\end{algorithmic}
\end{algorithm}
Recall that the trace of the Nyström approximation $\bm{X}(\bm{\Omega}^T \bm{X})^{\dagger} \bm{X}^T$ equals $\tr\left((\bm{\Omega}^T \bm{X})^{\dagger} (\bm{X}^T \bm{X})\right)$ via the cyclic property of the trace.

\subsection{Analysis of Nyström++}

In the following, we show that Algorithm~\ref{alg:npp} enjoys the same theoretical guarantees as 
Algorithm~\ref{alg:hpp}~\cite[Theorem 1.1]{hpp}. 
We begin with a result on the Frobenius norm error of the Nyström approximation.
\begin{lemma}\label{lemma:nystrom_bound}
Let $\bm{A} \in \mathbb{R}^{n \times n}$ be symmetric positive semidefinite and let $\bm{\Omega} \in \mathbb{R}^{n \times 2k}$ be a standard Gaussian matrix with $k \geq 5$. Then
\begin{equation*}
     \|\bm{A}-\bm{A}\bm{\Omega}(\bm{\Omega}^T\bm{A}\bm{\Omega})^{\dagger} \bm{\Omega}^T\bm{A}\|_F \leq \frac{542}{\sqrt{k}}\tr(\bm{A})
\end{equation*}
holds with probability at least $1-6e^{-k}$.
\end{lemma}
\noindent The proof of Lemma~\ref{lemma:nystrom_bound} builds on the following result.
\begin{lemma}[{\cite[Theorem 3]{gittensmahoney}}] \label{lemma:structural}
For a symmetric positive semidefinite matrix $\bm{A} \in \mathbb{R}^{n \times n}$ of rank at least $k$, let $\bm{A} = \bm{U} \bm{\Lambda} \bm{U}^T$
be a spectral decomposition  with the eigenvalues in non-increasing order on the diagonal of $\bm{\Lambda}$. Partition $\bm{U} = \begin{bmatrix} \bm{U}_1 & \bm{U}_2 \end{bmatrix}$ such that $\bm{U}_1 \in \mathbb{R}^{n \times k}$ and $\bm{\Lambda} = \begin{bmatrix} \bm{\Lambda}_1 & \\ & \bm{\Lambda}_2 \end{bmatrix}$ such that $\bm{\Lambda}_1 \in \mathbb{R}^{k \times k}$. Let $p \geq 1$ be an oversampling parameter and let $\bm{\Omega} \in \mathbb{R}^{n \times (k+p)}$ be such that $\bm{\Psi}_1 := \bm{U}_1^T \bm{\Omega}$ has full rank and define 
$\bm{\Psi}_2 := \bm{U}_2^T \bm{\Omega}$.
Then
\begin{equation}\label{eq:structural}
    \|\bm{A}-\bm{A}\bm{\Omega}(\bm{\Omega}^T\bm{A}\bm{\Omega})^{\dagger} \bm{\Omega}^T\bm{A}\|_F \leq \|\bm{\Lambda}_2 \|_F + \|\bm{\Lambda}_2^{1/2}\bm{\Psi}_2 \bm{\Psi}^{\dagger}_1\|_2\left(\sqrt{2\|\bm{\Lambda}_2\|_*} + \|\bm{\Lambda}^{1/2}_2\bm{\Psi}_2 \bm{\Psi}^{\dagger}_1\|_F\right).
\end{equation}
\end{lemma}
\begin{proof}[Proof of Lemma ~\ref{lemma:nystrom_bound}]
By proceeding as in the beginning of the proof of \cite[Lemma 7]{gittensmahoney} with probability at least $1-3e^{-k}$ we have\footnote{In the setting of~\cite[Lemma 7]{gittensmahoney}, set the quantities $p = k, t = e, u = \sqrt{2k}$ and $\bm{D} = \bm{\Lambda}_2^{1/2}$ to obtain~\eqref{eq:2norm_bound} and~\eqref{eq:frobeniusnorm_bound}.}
\begin{align}
    \|\bm{\Lambda}_2^{1/2} \bm{\Psi}_2 \bm{\Psi}_1^{\dagger}\|_2 &\leq \|\bm{\Lambda}_2^{1/2}\|_2\left(\sqrt{\frac{3k}{k+1}}e + \frac{2e^2k}{k+1}\right) + \|\bm{\Lambda}_2^{1/2}\|_F \frac{e^2 \sqrt{2k}}{k+1} \nonumber \\ 
    &\leq \sqrt{\|\bm{\Lambda}_2\|_2} (\gamma_1 + \gamma_2) + \sqrt{\|\bm{\Lambda}_2\|_*} \frac{\gamma_3}{\sqrt{k}} \label{eq:2norm_bound}
\end{align}
by letting $\gamma_1 := \sqrt{3}e, \gamma_2 := 2e^2$ and $\gamma_3 := \sqrt{2}e^2$, where $\|\cdot\|_*$ denotes the nuclear norm defined in Section~\ref{sec:notation}. Similarly, we have with probability at least $1-3e^{-k}$
\begin{align}
    \|\bm{\Lambda}_2^{1/2} \bm{\Psi}_2 \bm{\Psi}_1^{\dagger}\|_F &\leq \|\bm{\Lambda}_2^{1/2}\|_F \sqrt{\frac{3k}{k+1}} t + \|\bm{\Lambda}_2^{1/2}\|_2 \frac{2e^2k}{k+1} \nonumber\\
    & \leq \sqrt{\|\bm{\Lambda}_2\|_*}\gamma_1 + \sqrt{\|\bm{\Lambda}_2\|_2} \gamma_2. \label{eq:frobeniusnorm_bound}
\end{align}
By the union bound, both~\eqref{eq:2norm_bound} and~\eqref{eq:frobeniusnorm_bound} hold simultaneously with probability at least $1-6e^{-k}$.

$\bm{\Psi}_1 \in \mathbb{R}^{k \times 2k}$ is a standard Gaussian matrix and therefore has full row rank almost surely. We may therefore apply Lemma~\ref{lemma:structural} combined with the bounds~\eqref{eq:2norm_bound} and~\eqref{eq:frobeniusnorm_bound}. Hence, with probability at least $1-6e^{-k}$ we have 
\begin{align*}
    &\|\bm{A}-\bm{A}\bm{\Omega}(\bm{\Omega}^T\bm{A}\bm{\Omega})^{\dagger} \bm{\Omega}^T\bm{A}\|_F \\
    \leq & \|\bm{\Lambda}_2 \|_F + \|\bm{\Lambda}_2^{1/2}\bm{\Psi}_2 \bm{\Psi}^{\dagger}_1\|_2\left(\sqrt{2\|\bm{\Lambda}_2\|_*} + \|\bm{\Lambda}^{1/2}_2\bm{\Psi}_2 \bm{\Psi}^{\dagger}_1\|_F\right)\\
    \leq& \|\bm{\Lambda}_2\|_F + \left(\sqrt{\|\bm{\Lambda}_2\|_2} (\gamma_1 + \gamma_2) + \sqrt{\|\bm{\Lambda}_2\|_*} \frac{\gamma_3}{\sqrt{k}}\right)\left(\sqrt{\|\bm{\Lambda}_2\|_*}(\sqrt{2}+\gamma_1) + \sqrt{\|\bm{\Lambda}_2\|_2} \gamma_2\right)\\
    = & \|\bm{\Lambda}_2\|_F + \sqrt{\|\bm{\Lambda}_2\|_2 \|\bm{\Lambda}_2\|_*}\tilde \gamma_1 
      + \|\bm{\Lambda}_2\|_2 \tilde \gamma_2 + \|\bm{\Lambda}_2\|_* \frac{\tilde \gamma_3}{\sqrt{k}} \\
\leq & \frac{1+\tilde \gamma_1+ \tilde \gamma_2 + \tilde \gamma_3}{\sqrt{k}} \tr(\bm{A})
\end{align*}
where we set 
\[
 \tilde \gamma_1 := (\gamma_1 + \gamma_2)(\sqrt{2} + \gamma_1) + \frac{\gamma_2\gamma_3}{\sqrt{k}},
 \ 
 \tilde \gamma_2 := (\gamma_1 + \gamma_2)\gamma_2,
 \ 
 \tilde \gamma_3 := \gamma_3(\sqrt{2} + \gamma_1)
\]
and use the norm inequalities
\begin{align*}
 & \|\bm{\Lambda}_2\|_2 \leq \|\bm{\Lambda}_2\|_F \leq  \sqrt{\|\bm{\Lambda}_2\|_2\|\bm{\Lambda}_2\|_*} \leq \frac{1}{\sqrt{k}} \|\bm{\Lambda}\|_* = \frac{1}{\sqrt{k}} \tr(\bm{A}) \\
 & \|\bm{\Lambda}_2\|_*\leq \|\bm{\Lambda}\|_* = \tr(\bm{A})
\end{align*}
in the last step. The proof is completed by noting that 
$1+\tilde \gamma_1+ \tilde \gamma_2 + \tilde \gamma_3 \le 542$.
 \end{proof}
 We can now proceed to extend the main result on Hutch++~\cite[Theorem 1.1]{hpp} to Nyström++.
 \begin{theorem}\label{theorem:npp}
 Suppose that 
 Algorithm~\ref{alg:npp} (Nyström++) is executed with 
 $m = O\big(\varepsilon^{-1}\sqrt{\log(\delta^{-1})} + \log(\delta^{-1})\big)$ matrix-vector products and $\delta \in (0,1/2)$\footnote{This condition on $\delta$ allows us to bound all $\log(p\delta^{-1})$ terms that would otherwise appear in the proof (see e.g. Lemma~\ref{lemma:combinedbound_gaussian} where the term $\log(2\delta^{-1})$ appears) from above with $c\log(\delta^{-1})$ for some sufficiently large constant $c$.}. Then its output satisfies 
 \begin{equation}\label{eq:relerr_nystrom}
     |\tr_m^{\mathsf{n++}}(\bm{A}) - \tr(\bm{A})| \leq \varepsilon \tr(\bm{A})
 \end{equation}
 with probability at least $1-\delta$.
\end{theorem}
\begin{proof}
We follow the proof of~\cite[Theorem 1.1]{hpp}.
Let us first recall that 
$\bm{X} = \bm{A}\bm{\Omega}$, $\bm{Y} = \bm{A}\bm{\Phi}$ for $d\times m/2$ standard Gaussian random matrices $\bm{\Omega},\bm{\Phi}$ in Algorithm~\ref{alg:npp}. Throughout the proof, we assume that $m \ge c \log(\delta^{-1})$ for some (sufficiently large) constant $c$.

By Lemma~\ref{lemma:nystrom_bound}, there is a constant $C_1$ such  that
\begin{equation}
    \|\bm{A}-\bm{X}(\bm{\Omega}^T \bm{X})^{\dagger} \bm{X}^T\|_F \leq C_1 m^{-1/2} \tr(\bm{A}),
\end{equation}
with probability at least $1-\delta/2$. By Lemma~\ref{lemma:combinedbound_gaussian} there is a constant $C_2$ such that 
\begin{align*}
    &|\tr(\bm{A} - \bm{X}(\bm{\Omega}^T \bm{X})^{\dagger} \bm{X}^T) - \tr_{m/2}(\bm{A} - \bm{X}(\bm{\Omega}^T \bm{X})^{\dagger} \bm{X}^T)| \\
    &\leq C_2 m^{-1/2}\sqrt{\log(\delta^{-1})} \|\bm{A}-\bm{X}(\bm{\Omega}^T \bm{X})^{\dagger} \bm{X}^T\|_F.
\end{align*}
with probability at least $1-\delta/2$. By the union bound it holds with probability at least $1-\delta$ that
\begin{align*}
    |\tr_m^{\mathsf{n++}}(\bm{A})-\tr(\bm{A})| & = \big|\tr(\bm{A} - \bm{X}(\bm{\Omega}^T \bm{X})^{\dagger} \bm{X}^T) - \tr_{m/2}(\bm{A} - \bm{X}(\bm{\Omega}^T \bm{X})^{\dagger} \bm{X}^T)\big|\\
    &\leq  C_2 m^{-1/2}\sqrt{\log(\delta^{-1})} \|\bm{A}-\bm{X}(\bm{\Omega}^T \bm{X})^{\dagger} \bm{X}^T\|_F\\
    &\leq  C_1 C_2 m^{-1}\sqrt{\log(\delta^{-1})} \tr(\bm{A}).
\end{align*}
Hence, setting $m = O\big(\varepsilon^{-1}\sqrt{\log(\delta^{-1})} + \log(\delta^{-1})\big)$ implies the claim.
\end{proof}

\subsection{Adaptive Nyström++}
It is natural to aim at designing an adaptive version of Nyström++. Following A-Hutch++ we would need to find the minimum of 
\begin{equation}\label{eq:N_nystrom}
    m(r) = r + C(\varepsilon,\delta) \|\bm{A}-\bm{A}_{\text{n}}^{(r)}\|_F^2,
\end{equation}
where $\bm{A}_{\text{n}}^{(r)}$ is the rank-$r$ Nyström approximation.
Such an adaptive version clearly does not fit the streaming model. Moreover, we lose another advantage of Nyström++, that it only needs to perform $r$ matrix-vector products with $\bm{A}$ to get a rank-$r$ approximation, compared to $2r$ for the randomized SVD. Since we cannot compute $\|\bm{A}-\bm{A}_{\text{n}}^{(r)}\|_F^2$ we would need to decompose this term as done in~\eqref{eq:inprob_expansion}. This yields $\|\bm{A}-\bm{A}_{\text{n}}^{(r)}\|_F^2 = \|\bm{A}\|_F^2 - 2\tr(\bm{A}\bm{A}_{\text{n}}^{(r)}) + \|\bm{A}_{\text{n}}^{(r)}\|_F^2$. However, evaluating the term $- 2\tr(\bm{A}\bm{A}_{\text{n}}^{(r)}) + \|\bm{A}_{\text{n}}^{(r)}\|_F^2$ depending on $r$ requires additional matrix-vector products with $\bm{A}$. In summary, there is little advantage of using such an adaptive version of Nyström. 


\subsection{Numerical results}\label{sec:exp-nystrom}

To deal with potential numerical instabilities due to the appearance of the pseudoinverse in the Nyström approximation in line 3 of Algorithm~\ref{alg:npp}, in our implementation we use~\cite[Algorithm 16]{martinsson2020randomized}. This algorithm computes an eigenvalue decomposition $\bm{U} \bm{\Sigma} \bm{U}^T$ of the Nyström approximation of $\bm{A} + \nu \bm{I}$, where $\nu$ is a small shift, without explicitly forming the Nyström approximation. Once the eigenvalue decomposition is obtained the algorithm removes the shift by setting $\bm{\Lambda} = \max\left\{0,\bm{\Sigma}-\nu \bm{I}\right\}$ and returns $\bm{U} \bm{\Lambda} \bm{U}^T$, in factored form, as the stabilized Nyström approximation. The shift is set as $\nu = \sqrt{n} \texttt{eps}(\|\bm{A}\bm{\Omega}\|_2)$, where $\texttt{eps}(x)$ returns the distance to the next larger double precision floating point number to $x\in \mathbb{R}$ and $\bm{\Omega}$ is as in Algorithm \ref{alg:npp}. For further details, we refer to \cite{martinsson2020randomized,tropp2017fixed}.

We compare Nyström++ with Hutch++ and Single Pass Hutch++. We consider $m = 12 + 48k$ for $k \in \{0,1,2,\ldots,20\}$ and for each value of $m$ we run Hutch++, Single Pass Hutch++ and Nyström++ $100$ times each. 
We run the experiments on the matrices from Section~\ref{section:synthetic_matrices}, Section~\ref{section:estrada}, and Section~\ref{section:inverse}.  Moreover, we create two matrices with exponential decay, i.e. $\bm{A} = \bm{Q}\bm{\Lambda}\bm{Q}^T \in \mathbb{R}^{5000 \times 5000}$ where $\bm{Q}$ is a random orthogonal matrix and $\bm{\Lambda}$ is the diagonal matrix with entries $\bm{\Lambda}_{ii} = \exp(-i/s) $ for $i=1, \ldots, 5000$, 
where $s$ is a parameter controlling the rate of the decay. We let $s = 10$ and $s= 100$. 

The results are displayed in Figures~\ref{fig:algebraic_nystrom},~\ref{fig:estrada_nystrom},~\ref{fig:inv_nystrom}, and~\ref{fig:exponential_nystrom}, respectively. 
In each figure, the blue line is the average relative error from Nyström++, the red line is the average relative error from Hutch++ and the green line is the average relative error from Single Pass Hutch++. The shaded blue area shows the $10^{\text{th}}$ to $90^{\text{th}}$ percentiles of the results from Nyström++, and the shaded red area shows the $10^{\text{th}}$ to $90^{\text{th}}$ percentiles of the results from Hutch++. 

In all cases we observe that Single Pass Hutch++ is the weakest alternative. Moreover, in many cases Hutch++ and Nyström++ have similar performances, and in some cases Nyström++ outperforms Hutch++, see e.g. Figure \ref{fig:exponential_nystrom}.


\begin{figure}
\begin{subfigure}{.5\textwidth}
  \centering
  \includegraphics[width=\linewidth]{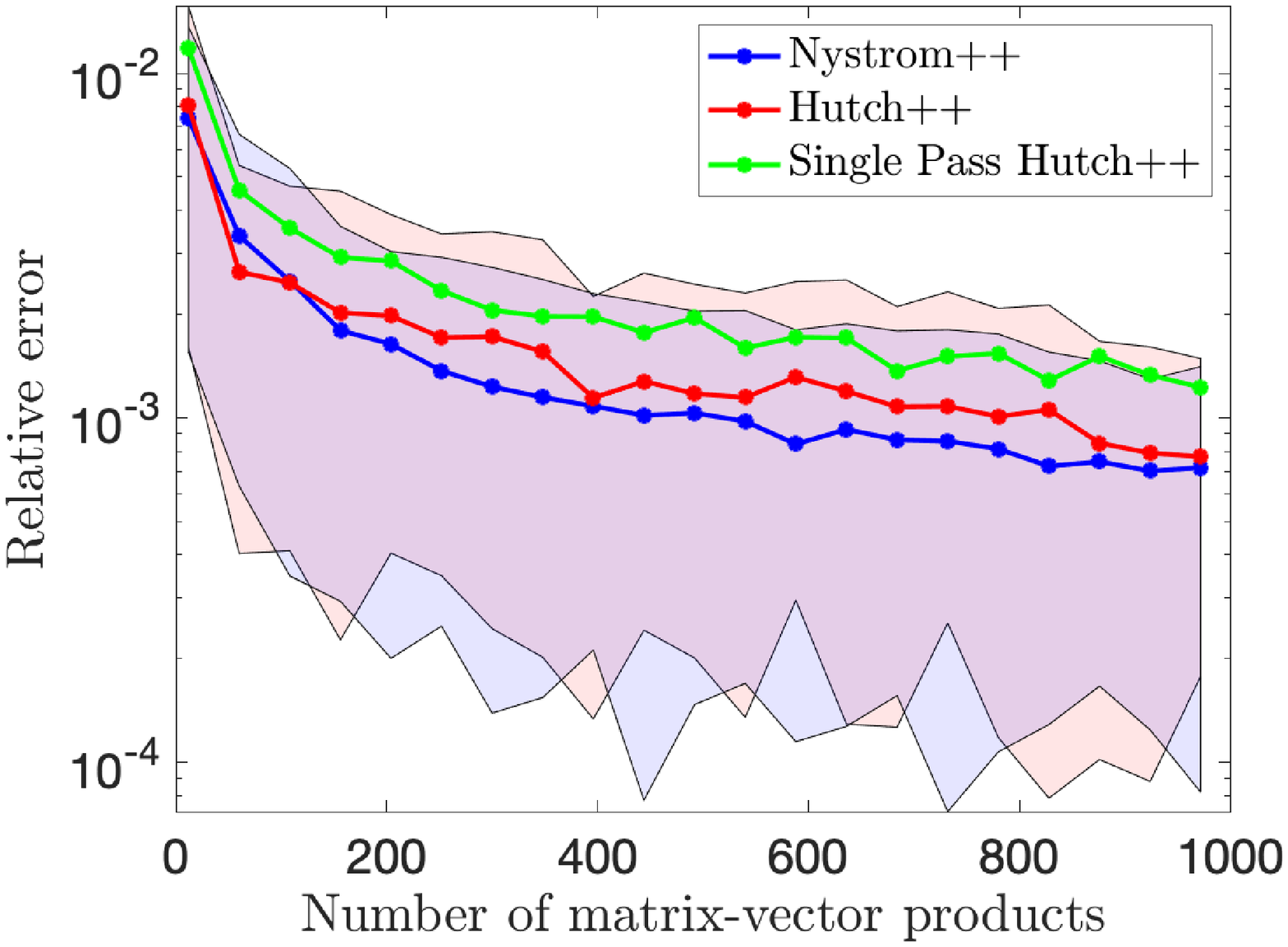}  
  \caption{$c = 0.1$}
  \label{fig:algebraic_c=01_nystrom}
\end{subfigure}
\begin{subfigure}{.5\textwidth}
  \centering
  \includegraphics[width=\linewidth]{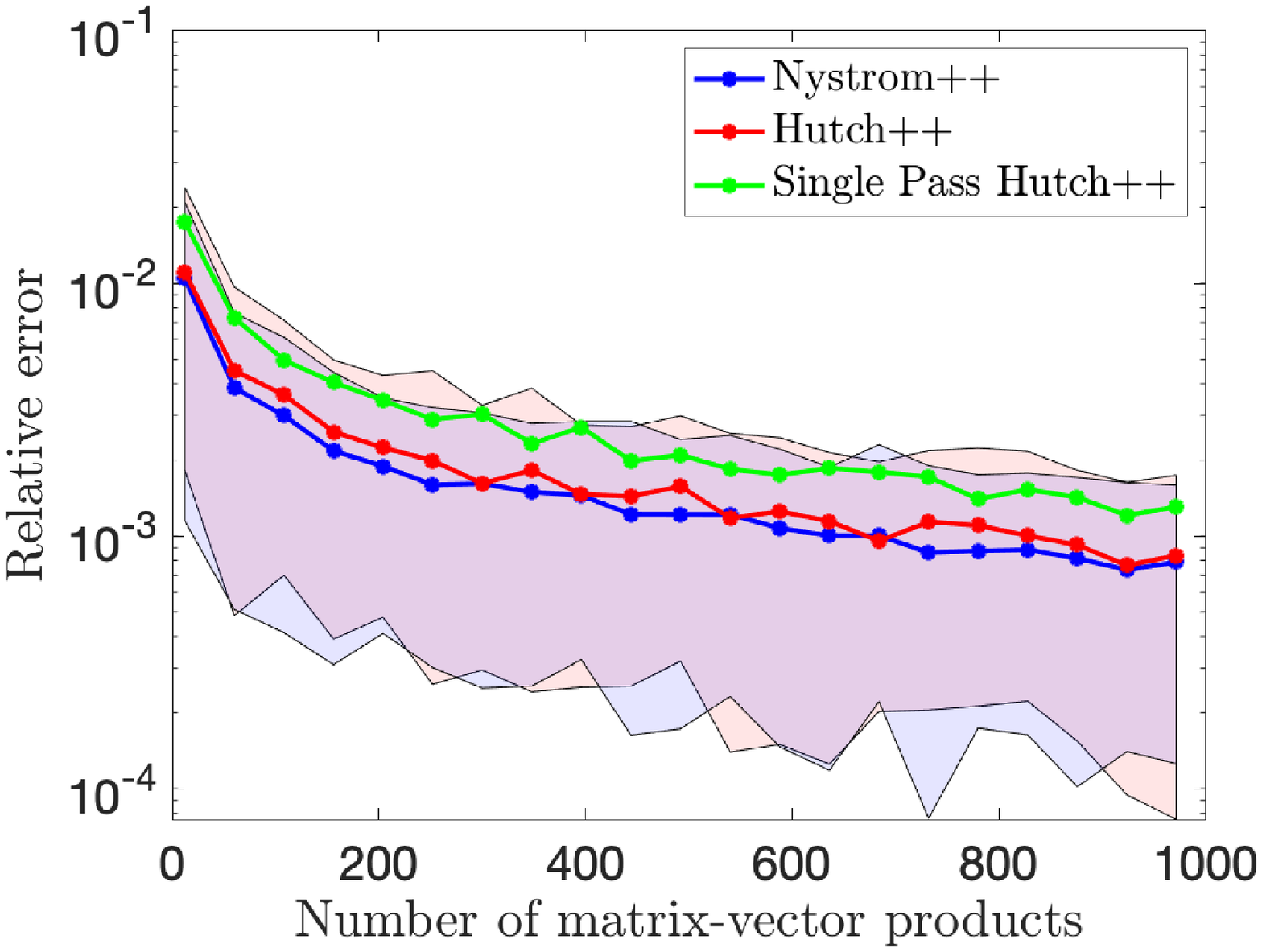}  
  \caption{$c=0.5$}
  \label{fig:algebraic_c=05_nystrom}
\end{subfigure}


\begin{subfigure}{.5\textwidth}
  \centering
  \includegraphics[width=\linewidth]{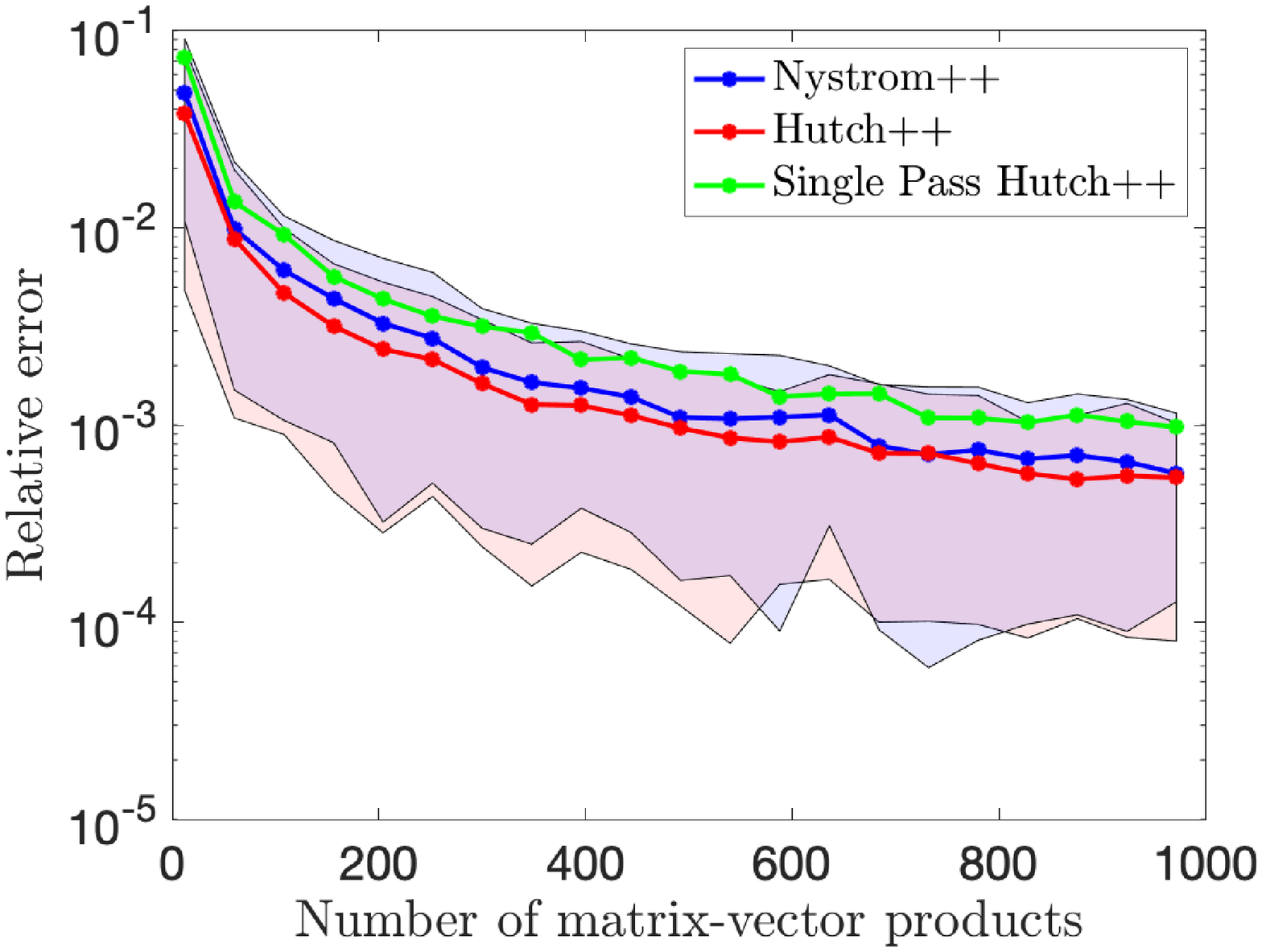}  
  \caption{$c=1$}
  \label{fig:algebraic_c=1_nystrom}
\end{subfigure}
\begin{subfigure}{.5\textwidth}
  \centering
  \includegraphics[width=\linewidth]{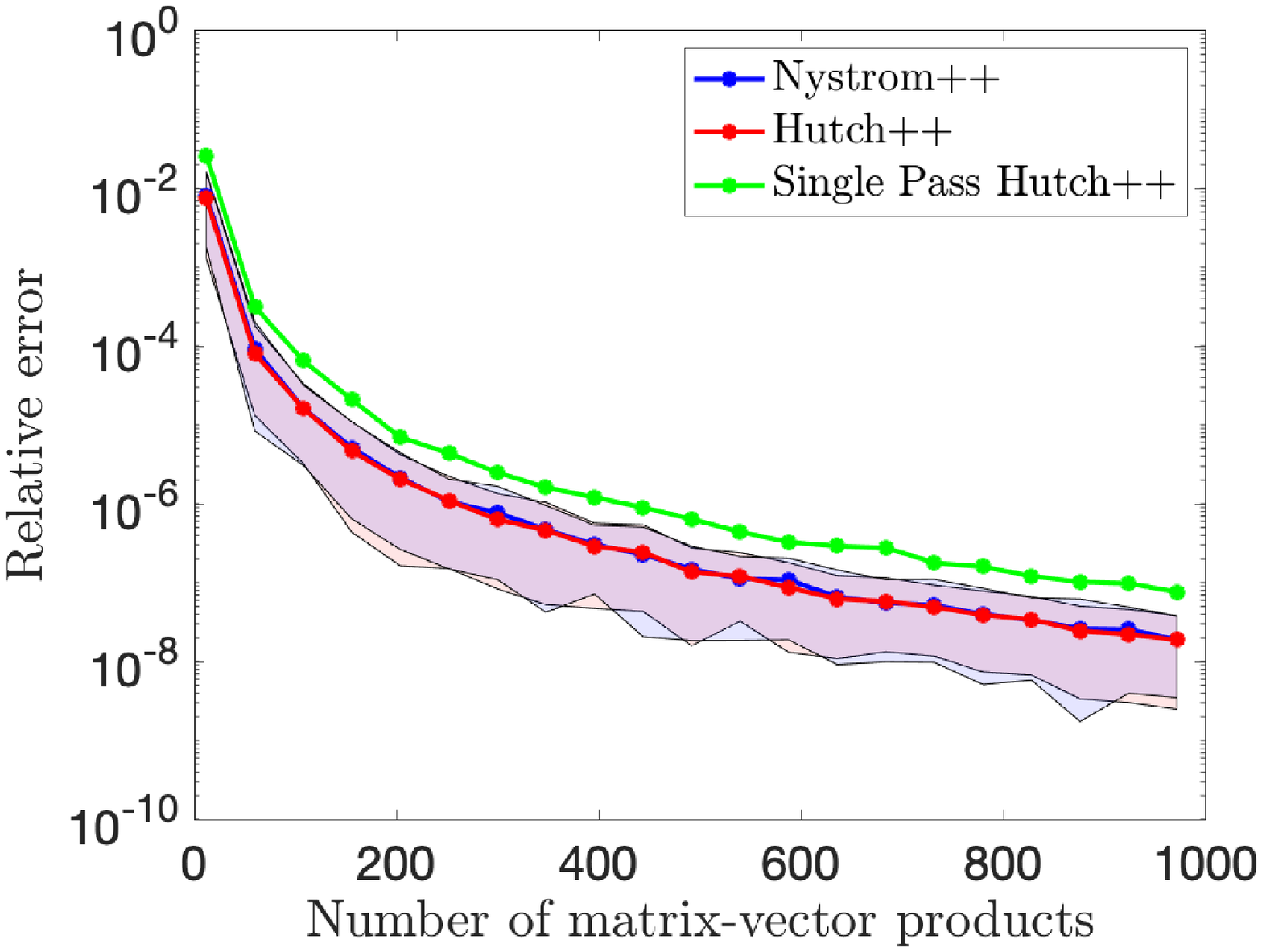}  
  \caption{$c=3$}
  \label{fig:algebraic_c=3_nystrom}
\end{subfigure}
\caption{Comparison of Hutch++, Single Pass Hutch++ and Nyström++ for the estimation of the trace of the synthetic matrices with algebraic decay described in Section~\ref{section:synthetic_matrices}.}
\label{fig:algebraic_nystrom}
\end{figure}

\begin{figure}
  \centering
  \includegraphics[width=.5\linewidth]{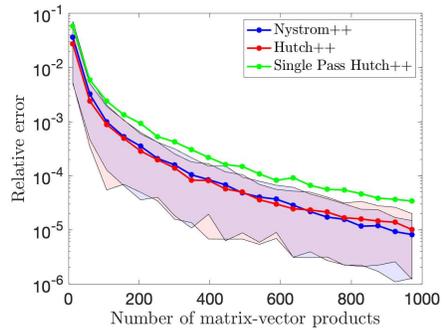}  
\caption{Comparison of Hutch++, Single Pass Hutch++ and Nyström++ for the estimation of the Estrada index as described in Section~\ref{section:estrada}.}
\label{fig:estrada_nystrom}
\end{figure}

\begin{figure}
\begin{subfigure}{.5\textwidth}
  \centering
  \includegraphics[width=\linewidth]{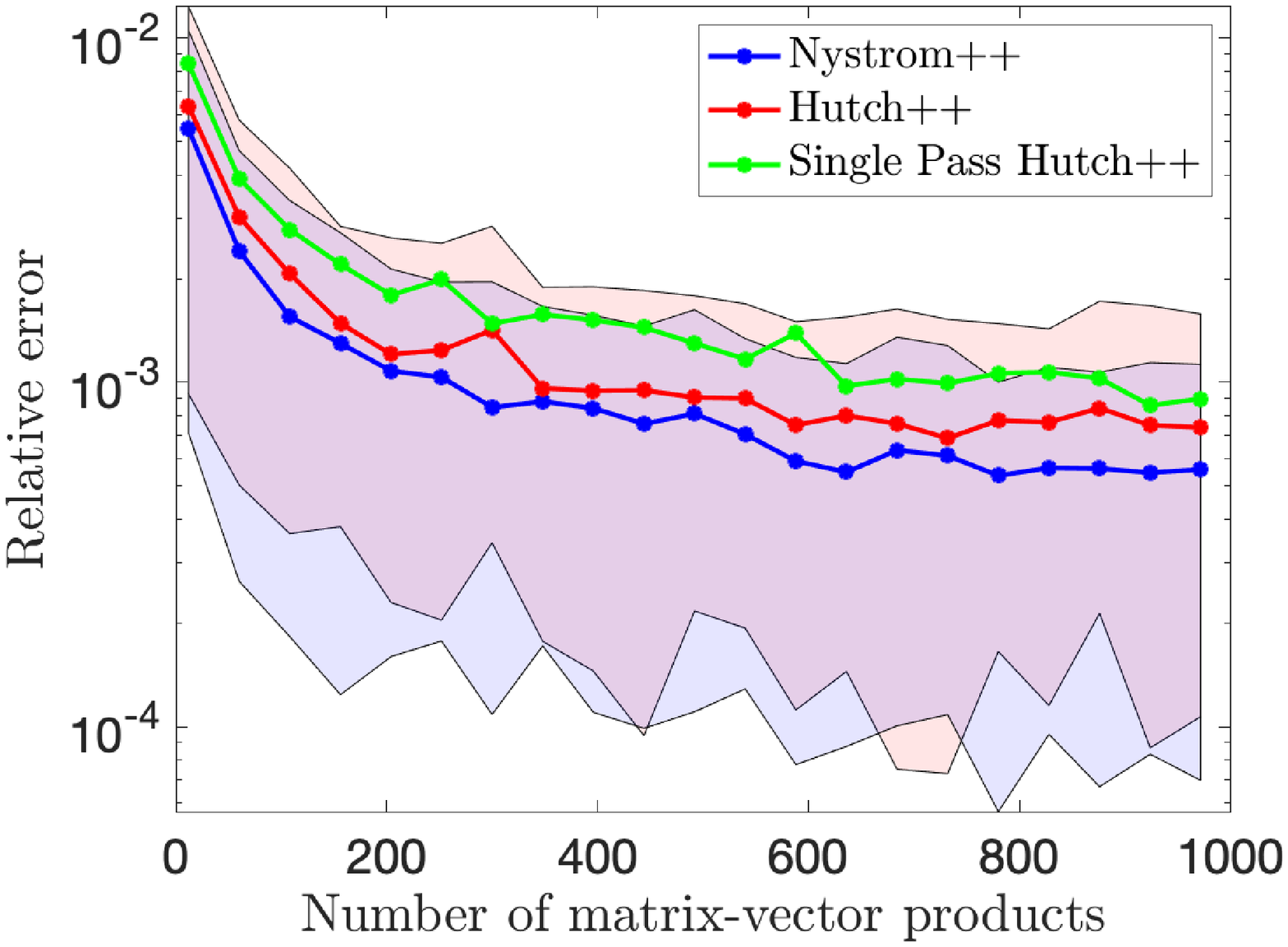}  
  \caption{Inverse of $\text{tridiag}(-1,4,-1)$.}
  \label{fig:tridiaginv_nystrom}
\end{subfigure}
\begin{subfigure}{.5\textwidth}
  \centering
  \includegraphics[width=\linewidth]{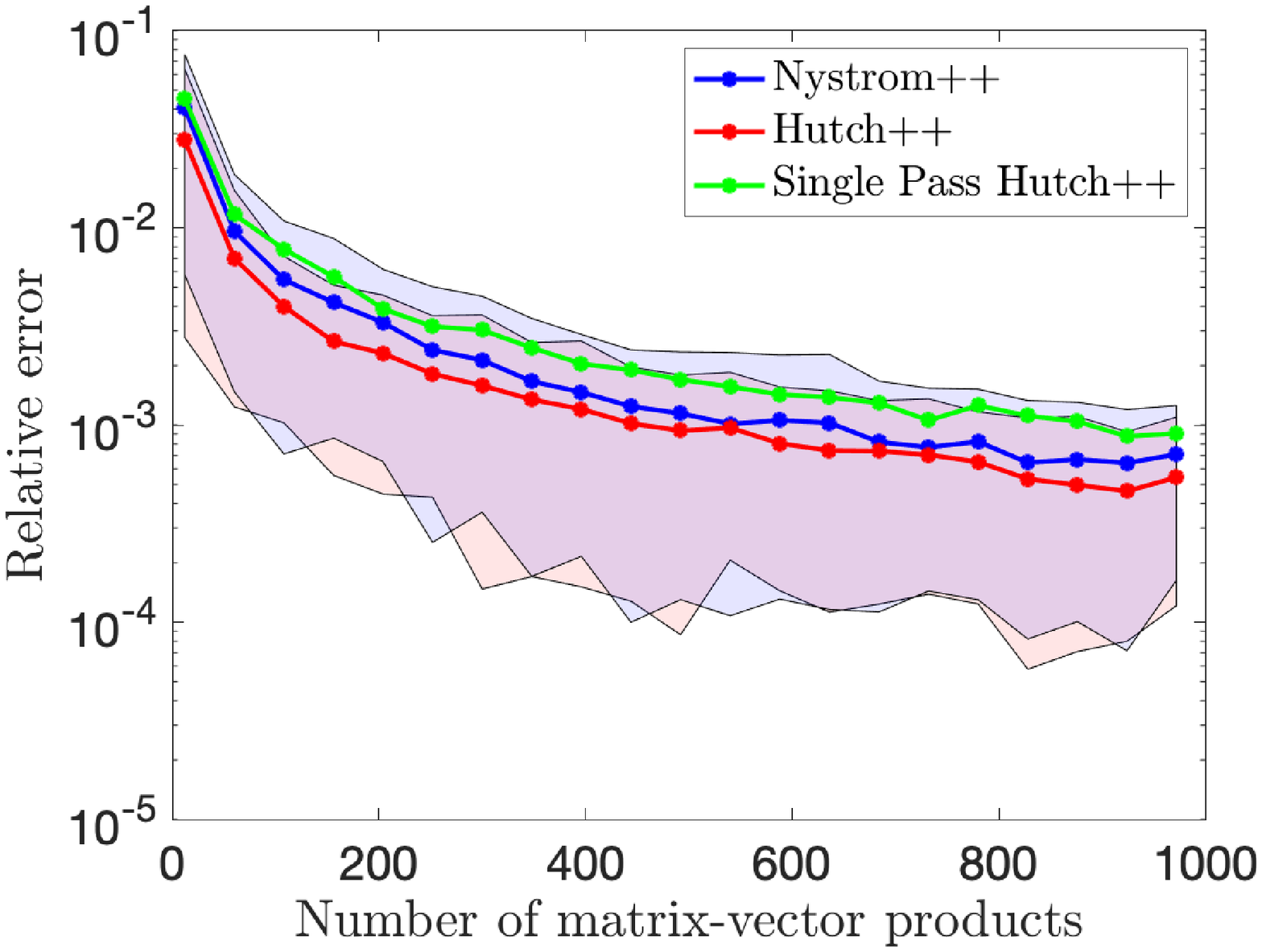}  
  \caption{Inverse of the matrix generated from discretizing Poisson's equation.}
  \label{fig:poissoninv_nystrom}
\end{subfigure}
\caption{Comparison of Hutch++, Single Pass Hutch++ and Nyström++ for the estimation of the trace of the inverse of the matrices described in Section~\ref{section:inverse}.}
\label{fig:inv_nystrom}
\end{figure}

\begin{figure}
\begin{subfigure}{.5\textwidth}
  \centering
  \includegraphics[width=\linewidth]{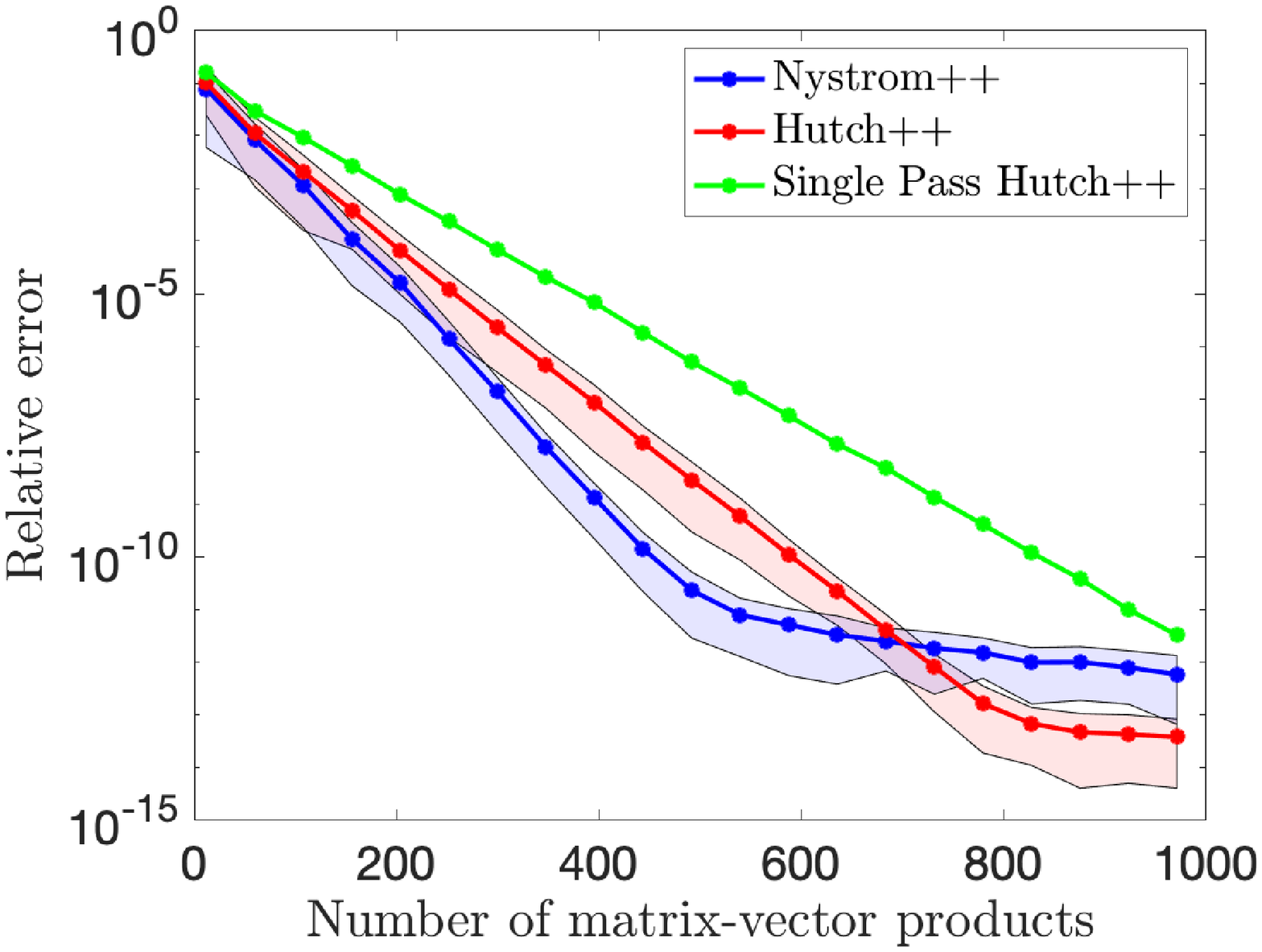}  
  \caption{$s=10$}
  \label{fig:exponential_s=10_nystrom}
\end{subfigure}
\begin{subfigure}{.5\textwidth}
  \centering
  \includegraphics[width=\linewidth]{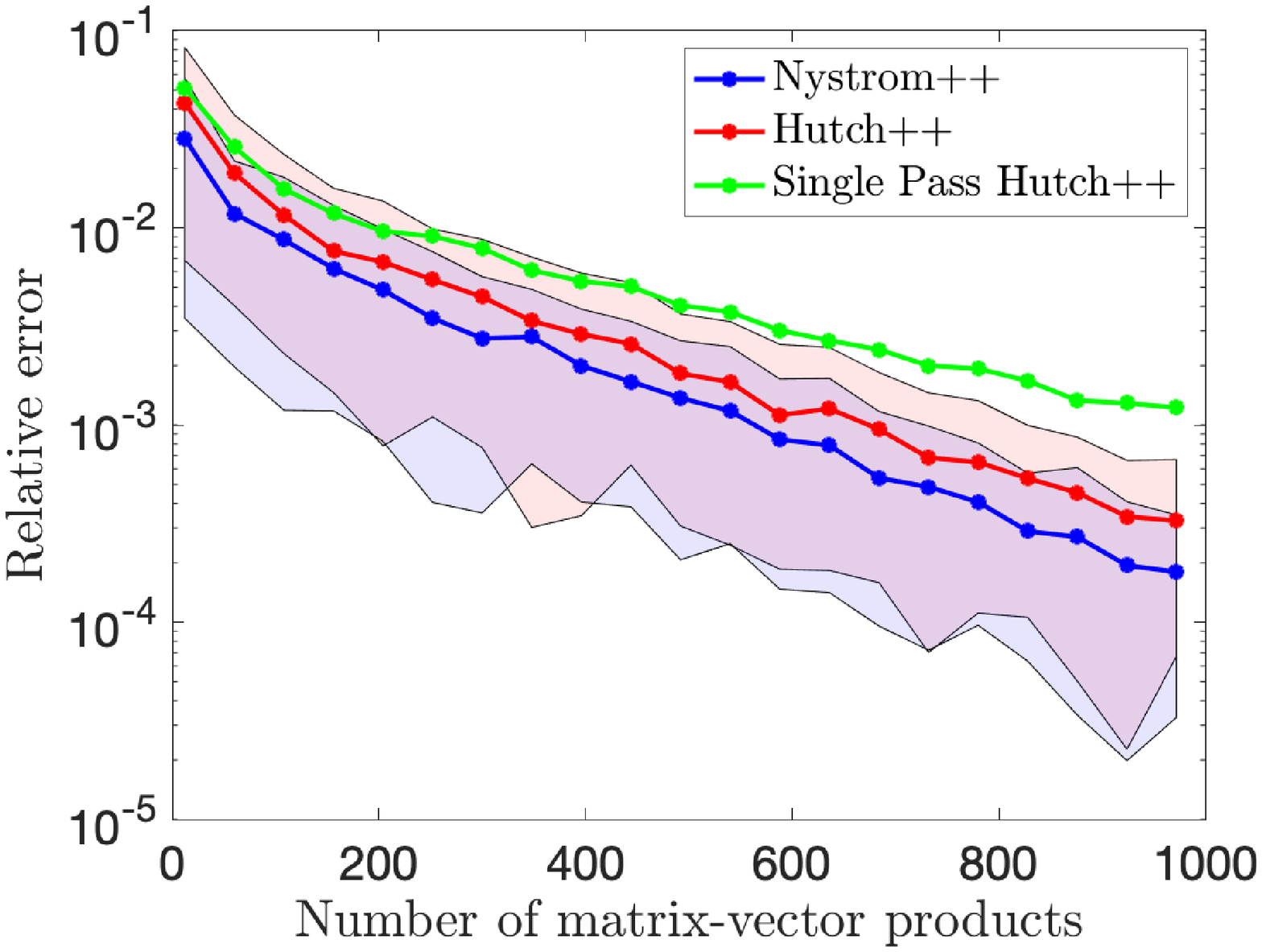}  
  \caption{$s=100$}
  \label{fig:exponential_s=100_nystrom}
\end{subfigure}
\caption{Comparison of Hutch++, Single Pass Hutch++ and Nyström++ for the estimation of the trace of the synthetic matrices with exponential decay described in Section~\ref{sec:exp-nystrom}.}
\label{fig:exponential_nystrom}
\end{figure}


\section{Conclusion}
We have presented an adaptive version of Hutch++, A-Hutch++, that will estimate the trace of a symmetric matrix $\bm{A}$ while attempting to minimize the number of matrix-vector products with $\bm{A}$ used overall. This algorithm also comes with the advantage that the user does not need to determine the number of matrix-vector products required to output an estimate of the trace that is within the prescribed error tolerance. We have tested A-Hutch++ on a variety of examples and we found that A-Hutch++ in many cases provided some improvement and in any case, it did not require more matrix-vector products compared to Hutch++ to achieve the same error. Furthermore, we presented a version of Hutch++ utilizing the Nyström approximation, which requires only one pass over the matrix. We proved that this algorithm satisfies the same theoretical guarantees of Hutch++. While this algorithm offers a similar performance as Hutch++, it performs significantly better than the previously proposed single pass algorithm Single Pass Hutch++.

\begin{appendices}


In this section we prove a version of the Hanson-Wright inequality and show that this is slightly stronger than the bound in~\cite[Theorem 1]{kressnercortinovis}. From this inequality we derive a version of Lemma \ref{lemma:combinedbound_gaussian}. We conclude with proving that this bound is asymptotically optimal.
\section{Hanson-Wright Inequality}\label{appendix:tailbound}
The Hanson-Wright Inequality we wish to prove is the following:
\begin{theorem}\label{theorem:gaussian}
Let $\bm{A} \in \mathbb{R}^{n \times n}$ be symmetric. Let $\bm{\omega}$ be a standard Gaussian vector of length $n$. Further, choose an arbitrary $c \in (0,1/2)$ and define $C = -\frac{1}{c}-\frac{\log(1-2c)}{2c^2}$. Then we have
\begin{equation*}
    \mathbb{P}\left(|\bm{\omega}^T \bm{A} \bm{\omega} - \tr(\bm{A})| \geq \varepsilon \right) \leq 2\exp\left(-\min\left\{\frac{\varepsilon^2}{4C\|\bm{A}\|_F^2}, \frac{c\varepsilon}{2\|\bm{A}\|_2}\right\}\right)
\end{equation*}
\end{theorem}
Theorem~\ref{theorem:gaussian} will be proved using tail bounds for sub-Exponential random variables. 
\begin{definition}[Sub-Exponential Random Variable]
A random variable $X$ is called sub-Exponential with parameters $\nu^2,\alpha >0$ if
\begin{equation*}
    \mathbb{E}\exp\left(t(X-\mathbb{E}X)\right) \leq \exp\left(\frac{t^2 \nu^2}{2}\right) \quad \text{for all } |t| \leq \frac{1}{\alpha}.
\end{equation*}
\end{definition}
For sub-Exponential random variables one has the following result, which follows from a Chernoff bound.
\begin{lemma}\label{lemma:sub-exponential}
(\cite[Proposition 2.9]{wainwright_hds}) If $X$ is a sub-Exponential random variable with parameters $(\nu^2,\alpha)$. Then
\begin{equation*}
    \mathbb{P}\left(|X-\mathbb{E}X| > \varepsilon\right) \leq 2 \exp\left(-\frac{1}{2}\min\left\{\frac{\varepsilon^2}{\nu^2}, \frac{\varepsilon}{\alpha}\right\}\right).
\end{equation*}
\end{lemma}
In order to prove Theorem~\ref{theorem:gaussian} we require the following 3 lemmas. They are proved using basic calculus techniques are therefore omitted.
\begin{lemma}\label{lemma:C}
If $C = -\frac{1}{c}-\frac{\log(1-2c)}{2c^2}$ for $c \in (0,1/2)$, then $C > 1$ and $\lim\limits_{c \rightarrow 0^+} C = 1$.
\end{lemma}
\begin{lemma}\label{lemma:bound1}
For $0 \leq x \leq c < \frac{1}{2}$ we have 
\begin{align*}
    \frac{\exp(-x)}{\sqrt{1-2x}} \leq \exp(Cx^2)
\end{align*}
where $C = - \frac{1}{c} - \frac{\log(1-2c)}{2c^2}$
\end{lemma}
\begin{lemma}\label{lemma:bound2}
For $x \geq 0$ we have $\frac{\exp(x)}{\sqrt{1+2x}} \leq \exp(x^2) \leq \exp(Cx^2)$ where $C$ as in Lemma \ref{lemma:bound1}. Furthermore, for $x \in (-1/2,0]$ we have $\frac{\exp(x)}{\sqrt{1+2x}} \geq \exp(x^2)$. In particular, for $x \in [0,1/2)$ we have $\frac{\exp(-x)}{\sqrt{1-2x}} \geq \exp(x^2)$.
\end{lemma}
We can now proceed to prove Theorem~\ref{theorem:gaussian}. 
\begin{proof}[Proof of Theorem \ref{theorem:gaussian}]
Let $X = \bm{\omega}^T \bm{A} \bm{\omega}$ where $\omega \sim N(\bm{0},\bm{I})$. We will show that $X$ is sub-Exponential with parameters $(2C\|\bm{A}\|_F^2,\frac{\|\bm{A}\|_2}{c})$. The final result will follow from Lemma~\ref{lemma:sub-exponential}.

Note that since $\bm{A}$ is symmetric we have
\begin{equation*}
    \bm{A} = \bm{Q} \bm{\Lambda} \bm{Q}^T
\end{equation*}
where $\bm{Q}$ is orthogonal and
\begin{equation*}
    \bm{\Lambda} = \begin{bmatrix} \bm{\Lambda}_+ & & \\
    & \bm{0} & \\
    & & \bm{\Lambda}_{-} \end{bmatrix}
\end{equation*}
where $\bm{\Lambda}_{\pm} = \text{diag}(\pm\lambda_1^{\pm},\cdots,\pm\lambda_{l_{\pm}}^{\pm})$ where $\lambda^{\pm}_{i} > 0 \quad \forall i = 1,\cdots,l_{\pm}$.\\\\
Note that by unitary invariance of Gaussian vectors we have
\begin{equation*}
    X-\mathbb{E}X = \bm{\omega}^T \bm{A} \omega-\tr(\bm{A}) \stackrel{d}{=} \sum\limits_{i=1}^{l_{+}} \lambda_i^{+}((\omega^{+}_i)^2-1) - \sum\limits_{j=1}^{l_{-}} \lambda_j^{-}((\omega^{-}_i)^2-1)
\end{equation*}
Hence,
\begin{equation*}
    \mathbb{E}\exp(t(X-\mathbb{E}X)) = \left(\prod\limits_{i=1}^{l_+} \mathbb{E} \exp(t\lambda_i^{+}((\omega^{+}_i)^2-1))\right)\left(\prod\limits_{j=1}^{l_-} \mathbb{E} \exp(-t\lambda_j^{-}((\omega^{-}_j)^2-1))\right)
\end{equation*}
Note that
\begin{equation}\label{eq:MGF}
    \mathbb{E}\exp(\pm t\lambda_k^{\pm}((\omega_k^{\pm})^2-1)) = \frac{\exp(\mp\lambda_k^{\pm}t)}{\sqrt{1\mp2\lambda_k^{\pm}t}}
\end{equation}
provided $\pm t \lambda_k^{\pm} < 1/2$.

By Lemma~\ref{lemma:bound1} and \ref{lemma:bound2} we have
\begin{equation*}
    \frac{\exp(\mp\lambda_k^{\pm}t)}{\sqrt{1-2\lambda_k^{+}t}} \leq \exp(C(\lambda_k^{\pm})^2 t^2), \quad |t\lambda_i^{\pm}| \leq c < 1/2
\end{equation*}
Hence, by inserting this into \eqref{eq:MGF} we get
\begin{align*}
    \mathbb{E}\exp(t(X-\mathbb{E}X)) &\leq \left(\prod\limits_{i=1}^{l_+}\exp(Ct^2(\lambda_i^{+})^2)\right)\left(\prod\limits_{j=1}^{l_-} \exp(Ct^2(\lambda_j^{-})^2)\right)\\
    & =\exp(Ct^2 \|\bm{A}\|_F^2)
\end{align*}
provided
\begin{equation*}
    |t| \leq \frac{c}{\|\bm{A}\|_2}
\end{equation*}
Hence, $X$ is sub-Exponential with parameters $\left(2C\|\bm{A}\|_F^2, \frac{\|\bm{A}\|_2}{c}\right)$.
\end{proof}
The following corollary can be proved via the diagonal embedding trick \cite[Theorem 1]{kressnercortinovis}.
\begin{corollary}\label{corollary:diagembedding}
Let $\bm{A}$ be as in Theorem \ref{theorem:gaussian}. Let $\tr_m(\bm{A})$ be the stochastic trace estimator with $m$ samples of i.i.d. standard Gaussian vectors. Then we have
\begin{align*}
    \mathbb{P}(|\tr_m(\bm{A})-\tr(\bm{A})| &\geq \varepsilon) \leq 2\exp\left(-m\min\left\{\frac{\varepsilon^2}{4C\|\bm{A}\|_F^2}, \frac{c\varepsilon}{2\|\bm{A}\|_2}\right\}\right)
\end{align*}
\end{corollary}
One can now show that Theorem~\ref{theorem:gaussian} is slightly stronger than the corresponding tailbound shown in~\cite[Lemma 4]{kressnercortinovis}. 
\begin{lemma}
For all $\bm{A}$ as in Theorem~\ref{theorem:gaussian} and $\varepsilon > 0$ there exists $c \in (0,1/2)$ such that 
\begin{equation*}
    \min\left\{ \frac{\varepsilon^2}{4C\|\bm{A}\|_F^2}, \frac{c\varepsilon}{2\|\bm{A}\|_2}\right\} > \frac{\varepsilon^2}{4(\|\bm{A}\|_F^2 + \varepsilon \|\bm{A}\|_2)}
\end{equation*}
\end{lemma}
\begin{proof}
Let $x = \frac{\|\bm{A}\|_F^2}{\varepsilon^2}$ and $y = \frac{\|\bm{A}\|_2}{\varepsilon}$. Hence, we need to show that there exists $c \in (0,1/2)$ such that 
\begin{equation}\label{eq:inequality}
    \min\left\{ \frac{1}{2Cx}, \frac{c}{y}\right\} > \frac{1}{2(x + y)}
\end{equation}
Note that for any $z > 0$ there is $c \in (0,1/2)$ such that $Cc = z$. Hence, choose $c$ such that $Cc = \frac{y}{2x} \Leftrightarrow x = \frac{y}{2Cc}$. This choice will guarantee \eqref{eq:inequality} since $1 > C(1-2c)$.
\end{proof}
One also has the following version of Lemma~\ref{lemma:combinedbound_gaussian}. 
\begin{lemma}\label{lemma:combinedbound_gaussian2}
Let $\bm{A} \in \mathbb{R}^{n \times n}$ be symmetric with stable rank $\rho(\bm{A})$. Let $\tr_m(\bm{A})$ be the stochastic trace estimator \eqref{eq:hutchinson} with $m$ matrix-vector multiplies with i.i.d. standard Gaussian random vectors. Let $c \in (0,\frac{1}{2})$ be arbitrary and define $C = -\frac{1}{c} - \frac{\log(1-2c)}{2c^2}$. Then, if $m \geq \frac{\log(2/\delta)}{c^2 C \rho(\bm{A})}$ we have that 
\begin{equation}
    |\tr_m(\bm{A})-\tr(\bm{A})| \leq 2 \sqrt{C} \sqrt{\frac{\log(2/\delta)}{m}}\|\bm{A}\|_F
\end{equation}
holds with probability at least $1-\delta$.
\end{lemma}
\begin{proof}
In the setting of Corollary~\ref{corollary:diagembedding} let $\varepsilon = 2\sqrt{C}\sqrt{\frac{\log(2/\delta)}{m}}\|\bm{A}\|_F$ and choose $m$ sufficiently large. Prooceed as in the proof of Lemma~\ref{lemma:combinedbound_gaussian}.
\end{proof}
\section{Tight constants}
In this section we prove that the smallest possible constant $\gamma$ such that
\begin{align*}
    |\tr_m(\bm{A})-\tr(\bm{A})| &\leq \gamma \sqrt{\frac{\log(2/\delta)}{m}}\|\bm{A}\|_F
\end{align*}
holds with probability at least $1-\delta$, is $\gamma=2$. If we let $c \rightarrow 0$ in Lemma \ref{lemma:combinedbound_gaussian2} we note that $C \rightarrow 1$. So we expect that we have an upper bound of $\gamma = 2$. Lemma \ref{lemma:tightness} implies that $\gamma = 2$ is the lower bound, and Lemma \ref{lemma:combinedbound_gaussian2} implies that it can be asymptotically reached. 
\begin{lemma} \label{lemma:tightness}
Let $\bm{A}$ be symmetric and $\gamma < 2$. Then, $\exists \delta^* > 0$ s.t. for all sufficiently large $N$ we have
\begin{align*}
    \mathbb{P}\left(|\tr_{N}(\bm{A})-\tr(\bm{A})| \leq \gamma \sqrt{\frac{\log(2/\delta^*)}{N}} \|\bm{A}\|_F\right)& < 1-\delta^*
\end{align*}
\end{lemma}
For this we need the following lemma
\begin{lemma}\label{lemma:infinity} Let $f:(0,1) \mapsto \mathbb{R}$ be a continuously differentiable function and $\lim\limits_{x \to 0^+} f(x) = 1$. Suppose $\lim\limits_{x \rightarrow 0^+}f'(x) = -\infty$. Then, $\exists \varepsilon > 0$ s.t. $x \in (0,\varepsilon) \Rightarrow f(x) < 1-x$.\end{lemma}
\begin{proof}
Note that $\lim\limits_{y \rightarrow 0^+} \int_y^xf'(t)dt = \int_0^x f'(t) dt = f(x)-1 \Rightarrow f(x) = 1 + \int_0^{x} f'(t) dt$ and the integral is understood as taking the limit to 0 at the lower bound. \\\\
Since $\lim\limits_{x \rightarrow 0^+} f'(x) = -\infty$ we know $\exists \varepsilon > 0$ s.t. $x \in (0,\varepsilon) \Rightarrow f'(x) < -2$. \\\\
Thus, $\forall x \in (0,\varepsilon)$ we have
\begin{align*}
    1-x - f(x) &= 1-x - 1 - \int_0^x f'(t) dt\\
    & = - \int_0^x (f'(t)+1) dt\\
    & \geq -\int_0^x(-2+1)dt = \int_0^x dt = x > 0
\end{align*}
as required.
\end{proof}
We now proceed with proving Lemma~\ref{lemma:tightness}.
\begin{proof}[Proof of Lemma \ref{lemma:tightness}]
In fact we prove something stronger: Let $S_N$ be the sample mean of $N$ i.i.d. random variables with mean 0 and standard deviation $\sigma$. Then, $\exists \delta^* > 0$ such that for all sufficiently large $N$ we have
\begin{equation}\label{eq:lemma_tightness}
    \mathbb{P}\left(|S_N| \leq \gamma \sqrt{\frac{\log(2/\delta^*)}{N}} \frac{\sigma}{\sqrt{2}}\right) < 1-\delta^*
\end{equation}
The result in Lemma \ref{lemma:infinity} immediately follows from (\ref{eq:lemma_tightness}). Define
\begin{equation*}
    p_N = \mathbb{P}\left(|S_N| \leq \gamma \sqrt{\frac{\log(2/\delta)}{N}} \frac{\sigma}{\sqrt{2}}\right)
\end{equation*}
By the Central Limit Theorem we have
\begin{equation*}
    p_N \rightarrow \Phi\left(\frac{\gamma\sqrt{\log(2/\delta)}\sigma}{\sqrt{2}\sigma}\right) - \Phi\left(-\frac{\gamma\sqrt{\log(2/\delta)}\sigma}{\sqrt{2}\sigma}\right) = \erf\left(\frac{\gamma}{2} \sqrt{\log(2/\delta)}\right)
\end{equation*}
as $N \rightarrow \infty$, where $\Phi$ is the cumulative distribution function of $N(0,1)$. Let $\nu = \frac{\gamma}{2} < 1$ and $p(\delta):= \erf(\frac{\gamma}{2} \sqrt{\log(2/\delta)})$. We will now show that $\exists \delta^* \in(0,1/2)$ such that \begin{align*}
    p(\delta) < 1-\delta^*
\end{align*}
It is easy to see that
\begin{align*}
    \lim\limits_{\delta \rightarrow 0^+} p(\delta) = 1 = 1-\delta |_{\delta = 0}
\end{align*}
For $\delta > 0$ we have
\begin{align*}
    p'(\delta) &= - \frac{\nu}{\sqrt{\pi}} \frac{\exp(-\nu^2 \log(2/\delta))} {\delta \sqrt{\log(2/\delta)}}\\
    & = - \frac{\nu}{\sqrt{\pi}} \left( \frac{\delta}{2}\right)^{\nu^2} \frac{1}{\delta\sqrt{\log(2/\delta)}} \\
    & = -\frac{\nu}{2^{\nu^2}\sqrt{\pi}} \delta^{\nu^2-1} \frac{1}{\sqrt{\log(2/\delta)}} \\
    & = -\frac{\nu}{2^{\nu^2}\sqrt{\pi}} \frac{1}{\delta^{\beta} \sqrt{\log(2/\delta)}}
\end{align*}
where $\beta = 1-\nu^2$. Since $\nu < 1$ we have $\beta > 0$ which implies that $\delta^{\beta} \sqrt{\log(2/\delta)} \rightarrow 0$ as $\delta \rightarrow 0^+$ and therefore $\lim\limits_{\delta \rightarrow 0^+}p'(\delta) = -\infty$. By Lemma \ref{lemma:infinity}, there is a neighbourhood $(0,\varepsilon)$ s.t. $p(\delta) < 1-\delta$ whenever $\delta \in (0,\varepsilon)$. Thus, there is a $\delta^* \in (0,\varepsilon)$ such that $p(\delta^*) < 1 - \delta^*$. Since, $\lim\limits_{N \rightarrow +\infty} p_N(\delta^*) = p(\delta^*)$ $\exists M \in \mathbb{N}$ s.t. $N > M \Rightarrow p_N(\delta^*) < 1 - \delta^*$. Then choose any $N > M$ will be sufficiently large.
\end{proof} 
\end{appendices}

\bibliographystyle{abbrv}

\end{document}